\documentclass[11pt]{amsart}
 \usepackage{amssymb,amsfonts, amsthm}
\usepackage{amscd, euscript, MnSymbol}
\usepackage{amscd, euscript}
\usepackage{pdfsync}
\usepackage{enumerate}
\usepackage{epsfig,graphicx,graphics} 
\usepackage{pdfsync}

\usepackage{cmmib57}
\newtheorem{definition}{{\bf Definition}}[section]
\newtheorem{theorem}[definition]{{\bf Theorem}}

\newtheorem{corollary}[definition]{{\bf Corollary}}
\newtheorem{proposition}[definition]{\noindent {\bf Proposition}}
\newtheorem{lemma}[definition]{\noindent {\bf Lemma}}
\newtheorem{sublemma}[definition]{\noindent {\bf Sublemma}}

\newtheorem{claim}[definition]{\noindent {\bf Claim}}

\newtheorem{problems}{Problems}

\usepackage{amsmath}

\providecommand{\ij}{}
\renewcommand{\ij}{{\{i,j\}}}		

\newcommand{\N}{\mathbb{N}}
\newcommand{\K}{\mathbb{K}}

\newcommand{\R}{\mathbb{R}}
\newcommand{\x}{\mathbf{x}}

\newcommand{\E}{\mathbf{E}}

\newcommand{\Pol}[1][n]{\K[\x_\ij]}		

\DeclareMathOperator{\Cov}{Cov}
\DeclareMathOperator{\Min}{Min}

\begin{document}\label{firstpage}

\title [A fixed point theorem for commuting families] {A FIXED POINT THEOREM FOR COMMUTING FAMILIES OF RELATIONAL HOMOMORPHISMS.
APPLICATIONS TO METRIC SPACES, ORIENTED GRAPHS AND ORDERED SETS} 

\author[A.Khamsi]{Amine Khamsi}
\address{Department of Mathematical Sciences, 
University of Texas at El Paso, 
El Paso, TX 79968, USA}
\email{mohamed@utep.edu}
\author[M.Pouzet] {Maurice Pouzet}
\address{Univ. Lyon, University Claude-Bernard  Lyon1, UMR 5208, Institut Camille Jordan, 
43, Bd. du 11 Novembre 1918,
69622 Villeurbanne, France et Department of Mathematics and Statistics, The University of Calgary, Calgary, Alberta, Canada}
\email{pouzet@univ-lyon1.fr}

\date{\today}

\maketitle

\begin{abstract}   We extend to binary relational systems the notion of compact and normal structure, introduced by  J.P.Penot for
metric spaces,  and we prove that for the involutive and reflexive ones, every commuting
family of relational homomorphisms has a common fixed point. The proof is based upon the clever argument
that J.B.Baillon discovered  in order to show that a similar conclusion  holds for bounded
hyperconvex metric spaces  and then refined by the first author to metric spaces with a compact and normal structure. Since the non-expansive mappings are relational homomorphisms, our result
includes those of T.C.Lim,  J.B.Baillon and the first author.  We show that it  extends the Tarski's fixed point theorem to 
 graphs which are retracts of reflexive oriented zigzags of bounded length. Doing so,  we illustrate the fact  that the consideration of binary relational systems or of generalized metric spaces are equivalent. 
\end{abstract}
AMS subject classification(2010). Primary: 05, 06, 08, Secondary: 37C25.

Key words: Fixed-point, non-expansive mappings, normal structure, Chebyshev's center, modular function  space, relational homomorphisms, order-preserving maps, metric spaces, hyperconvex spaces, relational systems, ordered sets, graphs, fences, zigzags, retracts.

\section{Introduction}

	Two results about fixed points are very much related. One is the famous Tarski's theorem (\cite{tarski}, 1955): \emph{every
order-preserving map on a complete lattice has a fixed point}. The other is a theorem of
R.Sine and P.M.Soardi (\cite{sine}, \cite{soardi}, 1979): \emph{every non-expansive mapping on a bounded hyperconvex metric space has a fixed point}. Indeed, as was shown by D.Misane and the second author (\cite{misane}, 1984, see also \cite{pouzet}, 1985 and  \cite{jawhari-al}, 1986)
if one considers a generalisation of metric spaces, where -instead of real numbers- the distance values
are members of an ordered monoid equipped with an involution, then the Sine-Soardi's theorem is still valid and for a particular
ordered monoid, these generalized metric spaces and their non-expansive mappings translate into 
ordered sets and order-preserving maps and --as a matter of fact--  hyperconvex spaces correspond to
complete lattices.
	
	Since A.Tarski obtained, in fact, that every commuting family of order-preserving maps on a 
complete lattice has a common fixed point, E. Jawhari et  al \cite{jawhari-al} considered the question whether in this frame every
commuting family of non-expansive mappings on a bounded hyperconvex space has a common fixed point,
discovering that it was still unsettled in the frame of ordinary metric spaces. They got a positive  answer for 
countable families; J.B.Baillon (\cite{baillon}, 1986)  got a positive answer for
arbitrary families acting on ordinary hyperconvex metric spaces.
The Baillon's proof is based upon a clever compactness argument. At firt glance, this argument  
works with minor changes for generalized  hyperconvex spaces considered in  \cite{jawhari-al} and, on an other hand, with some extra  work,  it can
be adapted to metric spaces endowed with a compact and normal structure --as abstractly defined by
J.P.Penot(\cite{penot}, 1977)-- spaces  which include the  hyperconvex ones. This extension was done by the first author in \cite{khamsi}.	

In this paper we propose a generalization of the Penot's notions in the frame of binary relational systems and their relational homomorphisms.
Indeed, on one hand, the non-expansive mappings   $f$ acting on an ordinary metric space, (or a
generalized one), say $(E, d)$, with distance function  $d$ from $E\times E$ into the set $\R^{+}$ of nonnegative  reals  (or into an ordered 
monoid 
$V$ equipped with an involution),  are relational homomorphisms of the binary relational system $\mathbf E:=(E, \{\delta_v : v\in V\})$, where $\delta_v:=\{ (x,y)\in E\times E : d(x,y)
\leq  v \}$ for every $v$ belonging to $V$. On an other hand, the Penot's notions are  very easy to define in this
frame. We prove that \emph{if a  reflexive and involutive binary relational system has a  compact and normal structure then 
every commuting family of relational homomorphisms has a common fixed point} (Theorem \ref{thm:cor}).  As an illustration, we get that \emph{on a graph which is a retract  of a product of  reflexive oriented zigzags of bounded length, every
commuting family of preserving maps has a common fixed point} (Theorem \ref{thm:cor4}); Tarski's result
corresponds to the case of a retract of a power  of  a two-element zigzag. Characterizations of  reflexive and involutive binary relational system with  a  compact and normal structure are left open.  This paper is an other opportunity to go beyond the analogy between metric spaces and binary relational sytems. We consider  generalized metric spaces whose values of the distance belong to an  involutive Heyting algebra (or involutive op-quantale) as it was initiated  in \cite {jawhari-al}. In this context, the notion of one-local retract, which is the key in proving our main result, fits naturally with the parent notion of hole-preserving map. Our illustration with graphs fits in the case of bounded hyperconvex spaces.   

After this introduction, this paper consists of four additional sections. Section 2  contains   the notions of compact and normal structure  for relational systems; an illustration with  a fixed-point result is given. Section 3 contains the notion of one-local retract. The main property,  Theorem \ref{thm:best} is stated;  Theorem \ref {thm:cor} is  given  a consequence. This property is proved in Section 4. Section 5 is an attempt to illustrate our main result. Subsection 5.2 contains the exact relationship between reflexive involutive binary systems and generalized metric spaces over an involutive monoid (e.g. Lemma \ref{lem:metric-relation}  and Theorem \ref{thm:cor2}). In Subsection 5.3, the notion of hyperconvexity is recalled. Notions of  inaccessibility   and boundedness insuring that hyperconvex spaces have a compact and normal structure are stated (Corollary \ref{cor:compact+normal}). Spaces over a Heyting  algebra  with their main properties are presented (Theorem \ref{thm:hyper1} and Theorem \ref{thm:hyper2}). Subsection 5.4 contains  the 
relationship between one-local retract and  hole-preserving maps. Subsection 5.5. rassembles the results for ordinary metric spaces. The case of ordered sets is treated in Subsection 5.6. It contains a characterization of posets with a compact structure. The case of directed graphs with the zigzag distance is treated in Subsection 5.7. It contains a  characterization of graphs isometrically embeddable into a product of oriented zigzags (Theorem \ref{theo:isometric}) and our fixed point theorem (Theorem \ref{thm:cor4}).

\section{Basic definitions, elementary properties and a fix-point result}
\subsection{Binary relations and metric notions}

We adapt to binary relations and to binary relational systems the basic notions of the theory of metric spaces. The trick we use for this purpose consists to denote by $d(x,y)\leq r$ the fact that the pair $(x,y)$ belongs to the binary relation $r$, and to interpret  $d$ as a distance,  $d(x,y)$  and $r$ as numbers (a justification is given in Subsection \ref{subsection:preserve}).

The basic concepts about relational systems are the following. For a set $E$, a \emph{binary relation} on  $E$ is
any  subset $r$  of $E\times E$; the \emph{restriction}  of $r$ to a subset $A$ of $E$ is $r_{\restriction A}:= r\cap (A\times A)$. The \emph{inverse} of $r$ is the binary relation $r^{-1}:= \{(x,y): (y,x)\in r\}$; the \emph{diagonal} is  $\Delta_E:= \{(x,x): x\in E\}$. A relation $r$ is \emph{symmetric} if $r=r^{-1}$; the relation $r$ is \emph{reflexive} if $\Delta_E\subseteq r$.  A   map $f: E\rightarrow E$ {\it preserves} $r$ if $(f(x), f(y)) \in  r$
whenever
$(x,y) \in r$; the map  $f$  \emph{preserves} a subset $A$ if $f(A)\subseteq A$ (this amounts to say that it preserves the unary relation $A$). The map $f$ {\it preserves}  a set $\mathcal E$ of binary relations $r$ on $E$ if it
preserves every member $r$ of $\mathcal E$.  The pair $\mathbf E:= (E, \mathcal E  )$ is a \emph{binary
relational system};  note that the maps which preserve $\mathcal E$  are in fact the relational homomorphisms  of this
system. We denote by $End (\mathbf E)$ the collection   of self-maps which preserve $\mathcal E$ (as far we only consider self-maps, no indexation of $\mathcal E$  by some index set is
required).  We set $\mathcal E^{-1}:= \{ r\subseteq E\times E: r^{-1}\in\mathcal E\}$ and we say that $\mathbf E:= (E, \mathcal E)$ is \emph{involutive} if $\mathcal E=\mathcal E^{-1}$; we say that $\mathbf E$ is \emph{reflexive}, resp. \emph{symmetric}, if each member $r\in \mathcal E$ is reflexive, resp. symmetric.  For a subset $A$ of $E$, the {\it restriction} of $\mathcal E$ to $A$ is 
$\mathcal {E}_{\restriction A}:= \{r_{\restriction A}: r\in \mathcal E\}$ and the {\it restriction} of $\mathbf E$ to $A$ is the binary
relational system $ \mathbf {E}_{\restriction  A}: =(A,  \mathcal E_{\restriction  A})$. For a subset $\mathcal E'$ of binary relations on $A$, we set $\bold p^{-1}_A(\mathcal E'):= \{r\in \mathcal E: r_{\restriction A}\in \mathcal E'\}$.  

	Let  $r$ be a binary relation on $E$ and let $x \in E$; the {\it ball} of {\it center } $x$, {\it
radius}  $r$, is the set  $B(x, r ):=\{ y \in E : (x,y) \in r\}$. Let $\mathcal E$ be a set of binary relations on $E$. We denote by $\mathcal B_{\mathcal E}$  the set of balls whose
radius belong to $\mathcal E$  that is  $\mathcal B_{\mathcal E} :=\{ B(x, r ) : x\in E, r\in \mathcal E\}$, we denote by $\hat {\mathcal B}_{\mathcal E}$  the set of all intersections of members of 
$\mathcal {B}_{\mathcal E}$, and we set $\hat {\mathcal B}_{\mathcal E}^*:= \hat {\mathcal {B}}_{\mathcal E} \setminus \{\emptyset \}$. Note that, as the intersection over the empty set,  $E\in \hat {\mathcal B}_{\mathcal E}$. For a subset $A$ of  $E$, the
$r$-{\it center}  is the set $C(A, r ):= \{ x\in E : A \subseteq  B(x, r )\}$.  We set $\Cov_{\mathcal E} (A):= \bigcap \{ B \in   \mathcal B_{\mathcal E}  : A\subseteq B\}$. The \emph{diameter} of $A$ is the set $\delta_{\mathcal E}(A):= \{r\in \mathcal E: A\times A  \subseteq r\}$; the {\it radius}  of
$A$ is  the set
$r_{\mathcal E}  (A):=\{ r \in  \mathcal E: A \subseteq B(x, r )$ for some $x\in  A \}$; note that $\delta_{\mathcal E}(\emptyset)= \mathcal E$ and $r_{\mathcal E}(\emptyset)= \emptyset$.  If $\mathbf {E}:= (E, \mathcal E)$, we may replace the index $\mathcal E$ in the previous notations by $\mathbf E$, e.g. $\mathcal {B}_{\mathbf E}$, $\hat {\mathcal B}^{\ast}_{\mathbf E}$, $\Cov_{\mathbf E}(A)$, $r_{\mathbf E}  (A)$ and $\delta_{\mathbf  E} (A)$ replace $\mathcal B_{\mathcal E}$, $\hat {\mathcal B}_{\mathcal E}^*$,  $\Cov_{\mathcal E}(A)$, $r_{\mathcal E}  (A)$ and  $\delta_{\mathcal  E} (A)$.

  Our notions of center and radius are inspired from the notions of Chebyshev's center and radius.

  The elementary properties about center, diameter and radius we need are given by the following
proposition:

\begin{proposition}\label{prop:center} Let  $\mathbf E  :=
(E, \mathcal E  )$ be a binary relational system,  $A\subseteq  E$ and  $r\subseteq E\times E$. Then:\begin{enumerate}[(i)]
\item $A \subseteq C(A, r)$ iff $r\in \delta  (A)$ and moreover $r\in  \delta (A)$ iff $r^{-1} \in \delta (A)$; 

\item $C(A, r )= \bigcap \{ B(a, r^{-1}) : a\in A\}$; 
\item If $r^{-1} \in  \mathcal  E$   then $C(A, r ) \in   \hat {\mathcal B}_{\mathbf E}$;  
\item 	$C(A, r ) = C ( \Cov_{\mathbf E} (A), r )$ whenever $r  \in   \mathcal E$;  
\item $r_{\mathbf E} (A)\subseteq r_{\mathbf E} (\Cov _{\mathbf E}(A))$ and if $A\not =\emptyset$, $\delta_{\mathbf E}(A) \subseteq r_{\mathbf E}(A)$;
\item 	 $\delta_{\mathbf E} (A)=  \delta_{\mathbf E} (\Cov_{\mathbf E}
(A) )$ provided that  $\mathbf E$   is involutive. 
\end{enumerate}
\end{proposition}
\begin{proof} 

\noindent $(i)$. Immediate. 

\noindent $(ii)$.   $x\in C(A, r)$ iff $A\subseteq B(x, r)$; this latter condition amounts to $x\in \{ B(a, r^{-1}) : a\in A\}$. 
\noindent $(iii)$. Follows immediately from $(ii)$. 
\noindent $(iv)$. From the definition of the $r$-center, $x\in C(A, r)$ means that $A\subseteq B(x, r)$; since $r \in \mathcal E$ this inclusion amounts to $\Cov_{\mathbf E}(A) \subseteq B(x,r)$. Again, from the definition of the $r$-center, this means $x\in C(\Cov_{\mathbf E}(A), r)$.

\noindent $(v)$. Let $r\in r_{\mathbf E} (A)$ then $C(A,r)\cap A\not = \emptyset$. Since $C(A, r ) = C ( \Cov_{\mathbf E} (A), r )$ from $(iv)$, we have $C(\Cov_{\mathbf E} (A),r)\cap \Cov_{\mathbf E}(A)\not = \emptyset$, hence $r \in  r_{\mathbf E} (\Cov_{\mathbf E}(A))$.  The second assertion is obvious. 

\noindent $(vi)$. Trivially $\delta_{\mathbf E}(\Cov_{\mathbf E}(A))\subseteq \delta_{\mathbf E}(A)$. Conversely,  let  $r \in \delta _{\mathbf E}Ž (A)$.  Then $A \subseteq  B(x, r )$ for every  $x \in Ž A$, that is  $A\subseteq   C (A, r)$. From $(iv)$,  this yields $A\subseteq C(\Cov_{\mathbf E} (A), r )$. Since $\mathbf E$ is involutive, $r^{-1}\in \mathcal E$, hence from $(ii)$ we have $C(\Cov_{\mathbf E} (A), r)\in \hat{\mathcal B}_{\mathbf E}$. Since $A\subseteq C(\Cov_{\mathbf E} (A), r )$ it follows $\Cov_{\mathbf E} (A) \subseteq  C(\Cov_{\mathbf E} (A), r)$ that is $r\in \delta _{\mathbf E} (\Cov_{\mathbf E}  (A) )$ by $(i)$. 
\end{proof}

\subsection{Compact normal structure and retraction}

We introduce the notion of compact and normal structure as Penot did for metric spaces (\cite{penot}, 1977) and we prove a fix-point result.   
  
\begin{definition}  A subset $A$ of a binary relational system $\mathbf E$ is  \emph{equally centered} if  $r_{\mathbf E} (A) =\delta_{\mathbf E}(A)$. 
\end{definition}
For an example, if $A$ is the empty set  and $\mathcal E$ nonempty then $A$ is not equally centered. If $A$  a singleton, say $\{a\}$, then $A$ is equally centered  (indeed, by $(v) $ of Proposition \ref{prop:center},  we have $\delta_{\mathbf E}(A)\subseteq r_{\mathbf E} (A)$; if $r \in r_{\mathbf E} (A)$ then $(x,x)\in r$ hence $r\in \delta_{\mathbf E} (A)$.  If in addition $\mathbf E$ is reflexive and involutive, $\Cov_{\mathbf E} (A)$ is equally centered. Indeed, by $(v) $ of Proposition \ref{prop:center} we have $\delta_{\mathbf E}(\Cov (A)\subseteq r_{\mathbf E} (\Cov (A))$; now, if $r\in r_{\mathbf E} (\Cov (A))$ then,  since $r$ is reflexive, $(x,x)\in r$ and thus $r\in \delta_{\mathbf E} (A)$; since $\delta_{\mathbf E} (A)=  \delta_{\mathbf E} (\Cov_{\mathbf E}
(A) )$ by  $(vi)$ of Proposition \ref{prop:center}, $r\in \delta_{\mathbf E} (\Cov_{\mathbf E}
(A) )$; hence $r_{\mathbf E} (\Cov(A)) =\delta_{\mathbf E}(\Cov(A))$ and $A$ is equally centered. A generalization of this fact is given in Lemma \ref{lem:compact}. 

\begin{definition}
A  binary relational system   $\mathbf E$ has a  \emph{normal  structure} if no $A\in Ž \hat{\mathcal B}_{\mathcal E}$ distinct from a singleton is equally centered. Equivalently, if  $|A|\not =­1$  then   $r_{\mathcal E} (A) \not =\delta_{\mathcal E}(A)$. 
\end{definition}

\begin{definition}A binary relational system $\mathbf E:= (E, \mathcal E)$  has a \emph{compact structure} if $\mathcal {B}_{\mathcal E}$  has the finite intersection property (f.i.p.) that is,  for every family  $\mathcal F$ of members of  $ \mathcal {B}_{\mathcal E}$,  the intersection  of $\mathcal F$   is nonempty  provided that the intersection of all finite subfamilies of $\mathcal F$ are nonempty. 
\end{definition}
As it is easy to see,  $\mathcal {B}_{\mathcal E}$ has the f.i.p. iff  $\hat {\mathcal B}_{\mathcal E}$  has the f.i.p.(if $\mathcal F$ is a family of members of  $\hat {\mathcal B}_{\mathcal E}$ associate the family $\mathcal G$ made of balls which contain some member of $\mathcal F$ and observe that $\bigcap \mathcal F= \bigcap \mathcal G$. If all finite intersection of members of $\mathcal F$ are nonempty, the same holds for $\mathcal G$. Hence,  if $\mathcal {B}_{\mathcal E}$ has the f.i.p.,  $\bigcap \mathcal G\not= \emptyset$. The equality $\bigcap \mathcal F= \bigcap \mathcal G$ yields  $\mathcal F\not=\emptyset$, thus $\hat {\mathcal B}_{\mathcal E}$  has the f.i.p.).

We  have trivially: 
\begin{lemma}  \label{lem:infimum} If a binary relational system $\mathbf E:= (E, \mathcal E)$ has  a  compact structure then every chain of members of $\hat {\mathcal B}^{\ast}_{\mathbf E}$ has an infinimum, namely the intersection of all members of that chain.  \end{lemma}

\begin{lemma}\label{lem:invariant}  Let $f$ be an endomorphism  of an involutive binary relational system  $\mathbf E:= (E, \mathcal E)$. Then:
\begin {enumerate} [{(i)}]
\item Every minimal member $A$ of $\hat {\mathcal B}_{\mathbf E}^*$ which is 
 preserved by  $f$ is equally centered;
\item  If $\mathbf {E}$ has a compact structure then every member of 
$\hat {\mathcal B }_{\mathbf E}^*$   preserved by  $f$ contains a minimal one.
\end{enumerate}
\end{lemma}
\begin{proof}$(i)$. Let $A\in \hat {\mathbf B}_{\mathcal E}^*$ and let $r\in r_{\mathbf E}(A)$, then $A':=C(A,r)\cap A$ is nonempty.  Indeed, by definition of $r_{\mathbf E}(A)$ there is some $x\in A$ such that $A\subseteq B_{\mathbf E}(x, r)$ and by definition of $C(A,r)$, $x\in C(A, r)$. This proves our assertion. Since $\mathbf E$ is involutive, $r^{-1}\in \mathcal E$, hence from $(iii)$ of Proposition \ref{prop:center}, $C(A,r)\in \hat {\mathcal B}_{\mathbf E}^{*}$, hence $A'\in \hat {\mathcal B}_{\mathbf E}^{*}$. Assuming that $f$  preserves $A'$ it follows $A=A'$ from the minimality of A. This means   $A\subseteq  C(A, r)$, that is  $r \in \delta _{\mathbf {E}}(A)$. Hence $r_{\mathbf {E}}(A)\subseteq \delta_{\mathbf E}(A)$. Since $\delta_{\mathbf E}(A)\subseteq r_{\mathbf {E}}(A)$, this  yields $r_{\mathbf E}(A)= \delta _{\mathbf {E}}(A)$.  Thus,  $A$ is equally centered as claimed. 
	
	In order to see that $f$  preserves $A'$, observes  that $f$ preserves $C(A, r)$. Indeed, first, since $f$ is a relational homomorphism  we have $f( C(A, r))\subseteq   C(f(A), r)$ ( for $x \in C(A, r)$  we have $A\subseteq B_{\mathbf E}(x, r)$  thus $f(A) \subseteq  B_{\mathbf E}(f(x), r)$  that is $ f(x)\in C(f(A), r))$). Next, from $(iv)$ of Proposition \ref{prop:center}, we have
 $C(f(A), r) = C(\Cov_{\mathbf E}(f(A)), r )$. To conclude, it suffices to prove that $\Cov_{\mathbf E}(f(A)) = A$. This assertion follows from  the minimality of $A$. Indeed, since  $f(A)\subseteq  A$,  we have $\Cov_{\mathbf E}(f(A))\subseteq \Cov_{\mathbf E} ( A) = A$;  it follows  $f(\Cov_{\mathbf E} ( f(A))) \subseteq f( A)\subseteq  \Cov_{\mathbf E}(f(A))$ that is $\Cov_{\mathbf E} (f(A))$ is preserved by $f$. The minimality of $A$ proves our assertion.

 $(ii)$. The fact that $\hat {\mathcal B}_{\mathbf E}$  has the f.i.p.  implies that $\hat {\mathcal B}^{\ast}_{\mathbf E}$, ordered by reverse of inclusion, is inductive (Lemma \ref{lem:infimum}). The subset of $\hat {\mathcal B}^{\ast}_{\mathbf E}$ made of $A$ such  that $f(A)\subseteq A$ is inductive too. The conclusion follows from Zorn's lemma. 
\end{proof}

\begin{corollary}\label{cor:main} If $\mathbf {E}$ involutive    has a compact and normal structure then every endomorphism  $f$ has a fixed point.
\end{corollary}

For metric spaces, this is the result of Penot (1979) extending the result of  Kirk (1965).

From Corollary \ref{cor:main},  one can derives:
\begin{proposition}\label{prop:main} If $\mathbf {E}$ involutive has a compact and normal structure then for every endomorphism $f$, the restriction $\mathbf E_{\restriction Fix(f)}$ to the set $Fix(f)$ of fixed points  of  $f$ has a compact and normal structure.
\end{proposition}

From this and the previous corollary, one deduces by an immediate recurrence that a finite set of commuting maps has a common fixed point. This leeds to the question of what happens with infinitely many. 

Behind the proof of the above proposition and the answer to the question is the   notion of one-local  retract. 
\section{One-local retracts and fixed points}

An map  
$g: E \rightarrow E$  is a \emph{retraction} of $\mathbf E$ if $g$ is an homomorphism of $\mathbf E$ such that $g\circ g =g$.  For a subset $A$ of $E$, we say  that $\mathbf E_{\restriction A}$  is  a \emph{retract} of $\mathbf E$ if $A$ is the image of $E$ by some retraction of $\mathbf E$. We say that $\mathbf E_{\restriction A}$  is a \emph{one-local retract} if for every $x\in E$,  $\mathbf E_{\restriction A}$ is  a retract of $\mathbf E_ {\restriction A\cup   \{x\}}$. 
\begin{lemma}\label{lem:one-local-retract} Let $\mathbf E:= (E, \mathcal E)$ be a binary relational system and $A$ be a subset of $E$. If $\mathbf {E}_{\restriction A}$  is  a one-local  retract then for every family of balls $B_{\mathbf E}(x_i, r_i )$, $x_i \in A$, $r_i\in \mathcal E$,  the intersection over $A$ is nonempty provided the intersection over $E$ is nonempty. The converse holds provided that $\mathbf E$ is reflexive and involutive; 
\end{lemma}

\begin{proof} Let $I$ be a set; let $\mathcal B:= \{B_{\mathbf E}(x_i, r_i ):  x_i \in A, r_i\in \mathcal E\}$ such that $B:=\bigcap \mathcal B\not = \emptyset$. Let $a\in B$ and let $h$ be a retraction from $\mathbf E_{\restriction A \cup \{a\}}$ onto $\mathbf E_{\restriction A}$. The map $h$ fixes $A$ and preserves the relations induced by $\mathcal E$ on $A\cup \{a\}$.  We claim that $h(a)\in B$. Indeed, let $i\in I$. Since $a\in B_{\mathbf E}(x_i, r_i)$ we have $(x_i, a)\in r_i$ and since $h$ preserves the relations induced by $\mathcal E$ on $A\cup \{a\}$,  $(h(x_i), h(a))\in r_i$. Since $h(x_i)=x_i$,  we get $(x_i, h(a))\in r_i$, hence $h(a)\in  B_{\mathbf E}(x_i, r_i)$. Our claim follows. 

Suppose that $\mathbf E$ is reflexive and involutive. We show that the ball's property stated in the lemma implies that $\mathbf E_{\restriction A}$ is a one-local retract. Let $a\in E \setminus A$. Let $\mathcal B:= \{B(u, r): u\in A, a\in B(u, r), r\in \mathcal E \}$. We have $a\in B:= \bigcap \mathcal B$ (note that $B=E$ if $\mathcal B= \emptyset$). Hence, $B\not=\emptyset$. According to the ball's property, $B\cap A\not =\emptyset$. Let $a'\in B\cap A$.  We claim that the map $h: A\cap \{a\}\rightarrow A$ which is the identity on $A$ and send $a$ onto $a'$ is a retraction  of $\mathbf E_{\restriction A\cup \{a\}}$. Since $h$ is the identity on $A$, it suffices to check that for every $r\in \mathcal E$ and $u\in A$:

\begin{enumerate} 

\item $(u, a)\in r$ implies $(u, a')\in r$;
\item $(a,u)\in r$ implies $(a',u)\in r$; 

\item $(a,a)\in r$ implies $(a',a')\in r$.

\end{enumerate}

The first item holds by our choice of $a'$; the second item is equivalent to  the first   because $\mathbf E$ is involutive and the third item holds because $\mathbf E$ is involutif. 
\end{proof}

\begin{lemma}\label{transitivity} Let $\mathbf E= (E, \mathcal E)$ be a binary relational system and $A\subseteq B \subseteq E$. 

\begin{enumerate}[{(i)}]
\item If $\mathbf E_{\restriction A}$ is a one-local retract of $\mathbf E$ then it is a one-local retract of $\mathbf E_{\restriction B}$.

\item If $\mathbf E_{\restriction A}$ is a one-local retract of  $\mathbf E_{\restriction B}$  and $\mathbf E_{\restriction B}$  is a  one-local retract  of $\mathbf  E$  then $\mathbf E_{\restriction A}$ is a one-local retract of $\mathbf E$. 

\end{enumerate}
\end{lemma}
\begin{proof}
The proof relies on the fact that $(\mathbf E_{\restriction B'})_{\restriction A'}= \mathbf E_{\restriction A'}$ for every $A'\subseteq B'$. 

$(i)$. Immediate.

$(ii)$. Let $x\in E\setminus A$. If $x\in B$ then,  since $\mathbf E_{\restriction A}$ is a one-local retract of $\mathbf E_{\restriction B}$, it is a retract of $\mathbf E_{\restriction A\cup \{x\}}$. If $x\not \in B$, then, since $\mathbf E_{\restriction B}$ is a one-local retract of $\mathbf E$, it is a retract of $\mathbf E_{\restriction B\cup \{x\}}$ by some map $g$. Let $y:= g(x)$. If $y\in A$, then $g_{\restriction A \cup \{x\}}$ is  a retraction of $\mathbf E_{\restriction A\cup \{x\}}$ onto $\mathbf E_{\restriction A}$. If $y\not \in A$, then,  since since $\mathbf E_{\restriction A}$ is a one-local retract of $\mathbf B$ it is a retract of $\mathbf E_{\restriction A\cup \{x\}}$ by some map $h$. The map $h\circ g$ is a retraction of $\mathbf E_{\restriction A\cup \{x\}}$ onto $\mathbf E_{\restriction A}$. 
\end{proof}

\begin{lemma}\label{lem:compact}  Let $\mathbf {E}_{\restriction X}$  be  a one-local retract of $\mathbf E$. If   $\mathbf  E$  has a compact structure then $\mathbf {E}_{\restriction X}$ too; if $\mathbf E$ is involutive and has a normal structure then $\mathbf {E}_{\restriction X}$ too.    
\end{lemma}
\begin{proof}
Let $\mathbf E':= \mathbf {E}_{\restriction X}$ and $\mathcal E':= \mathcal E_{\restriction X}:=\{ r\cap X\times X: r\in \mathcal E\}$. We prove the first assertion. Let $\mathcal B':= \{B_{\mathbf E'}(x'_i, r'_i ): i\in I, r'_i\in \mathcal {E'}\}$ be a family of balls of $\mathbf {E'}$ whose finite intersections are nonempty. For each $i\in I$, $r'_i=r_i\cap (X\times X)$ for some $r_i\in \mathcal E$.  The family $\mathcal B:= \{B_{\mathbf E}(x'_i, r_i ): i \in I \}$ of balls of $\mathbf E$ satisfies the f.i.p. hence has a nonempty intersection. Let $x\in \bigcap \mathcal B$.  A retraction $g$ from $\mathbf {E}_{\restriction X\cup \{x\}}$
onto $\mathbf {E}_{\restriction X}= \mathbf {E'}$ will send $x$ into $\bigcap {\mathcal B'}$, proving that this set is nonempty. 
We prove the second assertion.

 Let  $A\in Ž \hat{\mathcal B}^{\ast}_{\mathbf  {E'}}$.
We claim that:
\begin{equation}\label{equ:diam1}
\delta_{\mathbf E}(A)=\delta_{\mathbf E}(\Cov_{\mathbf E}(A)) 
 \end{equation}
 and 
\begin{equation}\label{equ:radi1}
r_{\mathbf E}(A)= r_{\mathbf E}(\Cov_{\mathbf E} (A)). 
\end{equation}

Indeed, equality (\ref{equ:diam1}) is item $(vi)$ of Proposition \ref{prop:center}.
Concerning equality (\ref {equ:radi1}), note that inclusion $r_{\mathbf E}(A)\subseteq  r_{\mathbf E}(\Cov_{\mathbf E}(A))$ is item $(v)$ of Proposition \ref{prop:center}. For the converse, let $r\in r_{\mathbf E}(\Cov_{\mathbf E}(A))$. Then, there is some  $x\in \Cov_{\mathbf E}(A)$ such that $\Cov_{\mathbf E} (A) \subseteq B_{\mathbf E}(x, r)$. Since $\mathbf E'$ is a one-local retract of $\mathbf E$,  there is a retraction of $\mathbf E_{\restriction X\cup \{x\}}$ onto $\mathbf E':=\mathbf E_{\restriction X}$ which fixes $X$. Let $a:= g(x)$. Since $A\subseteq \Cov_{\mathbf E}(A)\subseteq B_{\mathbf E}(x,r)$, $A\subseteq  B_{\mathbf E}(x,r)$; since $g$ fixes $A$, $A\subseteq B_{\mathcal E}(a, r)$, 					We claim that $a \in A$. Indeed, $A= \bigcap \{B_{\mathbf E'}(x'_i, r'_i) :i\in I\}$ with $x'_i\in X, r'_i\in \mathcal E_{\restriction X}$. For each $i\in I$, choose  $r_i\in \mathcal E$ such that $r'_i= r_i\cap X\times X$. Then $\Cov_{\mathbf E}(A)\subseteq  A_1:=\bigcap \{B_{\mathbf E}(x'_i, r_i) :i\in I\}$. Since $x\in \Cov_{\mathbf E}(A)$, $x \in A_1$. Since $g$ fixes each $x'_i$, $a:=g(x)\in A_1$. Since $a\in X$, $a\in A$, proving our claim. Hence $r\in  r_{\mathbf E}(A)$.

From these two equalities, we obtain:

\begin{equation}\label{equ:diam}
\delta_{\mathbf E}(\Cov(A))= \bold p^{-1}_X(\delta_{\mathbf E'}(A)) 
 \end{equation}
 and 
\begin{equation}\label{equ:radi}
r_{\mathbf E}(\Cov(A))= \bold p^{-1}_X(r_{\mathbf E'}(A)). 
\end{equation}
\noindent{\bf Proof of (\ref{equ:diam}).} By definition $\bold p^{-1}_X(\delta_{\mathbf E'}(A))= \{r\in \mathcal E: r_{\restriction X}\in \delta_{\mathbf E'}(A)\}= \{r\in \mathcal E: A\times A\subseteq r\}= \delta_{\mathbf E}(A)$ since $A\subseteq X$. Equality (\ref{equ:diam}) follows then from equality (\ref{equ:diam1}). 

\noindent{\bf Proof of (\ref{equ:radi}).} By definition, $\bold p^{-1}_X(r_{\mathbf E'}(A))= \{r\in \mathcal E: r_{\restriction X}\in r_{\mathbf E'}(A)\}= \{r\in \mathcal E: A\subseteq B_{\mathbf E'} (a', r_{\restriction X})\;  \text{for some}\;  a' \in A\}= \{r\in \mathcal E: A\subseteq B_{\mathbf E} (a', r)\; \text{for some}\;  a' \in A\}= r_{\mathbf E}(A)$. Equality (\ref{equ:radi}) follows from equality (\ref{equ:radi1}).

Suppose that $A$ is   equally centered in $\mathbf E'$, that is $\delta_{\mathbf E'}(A)=r_{\mathbf E'}(A)$.  From the equations above, we deduce that $\delta_{\mathbf E}(\Cov(A))
 = r_{\mathbf E}(\Cov (A))$. Hence $\Cov_{\mathbf E}(A)$ is equally centered. If $\mathbf E$ has a normal structure, 
$\vert \Cov_{\mathbf E}(A)\vert =1$, and since $A\subseteq \Cov_{\mathbf E}(A)$ and $A\not=\emptyset $,  $\vert A\vert=1$. Hence,  $\mathbf E'$ has a normal structure. 
\end{proof}

To prove Proposition \ref{prop:main} above, it suffices to prove:

\begin{proposition}\label{prop:main2} If $\mathbf {E}$ involutive    has a compact and normal structure then for every homomorphism $f$, the set $Fix(f)$ of fixed points of $f$ is a one-local retract of $\mathbf {E}$, thus $\mathbf E_{\restriction Fix(f)}$ has a compact and normal structure.
\end{proposition}

\begin{proof} 
	Let $B_{\mathbf E}(x_i, r_i )$, $x_i\in ŽFix(f)$,  such that  $A:= \bigcap_i‡ B_{\mathbf E}(x_i, r_i )$ is  nonempty. Since every $x_i$ belongs to $Fix(f)$, $f$ preserves $A$. According to $(ii)$ of Lemma \ref {lem:invariant}, since $A$ is an intersection of balls, $A$ contains an intersection of balls $A'$ which is equally centered and preserved by $f$. From the normality of $\mathbf E$,  $A'$  reduces to a single element, that is a fix-point of $f$.  Consequently,  $A\cap  Fix(f)\not= \emptyset$. According to Lemma \ref{lem:one-local-retract}, $Fix(f)$ is a one-local retract. 
\end{proof}

\begin{theorem}\label{thm:best} If $\mathbf E$  is involutive, reflexive and has a compact and normal structure then the intersection of every down-directed family of one-local retracts is a one-local retract.
\end{theorem}

We will prove Theorem \ref{thm:best} in the next section. From it, we derive easily our main result. 

\begin{theorem} \label{thm:cor}
If  $\mathbf E$  is involutive, reflexive and  has a compact and normal structure then every commuting family $\mathcal F$ of endomorphisms of $\mathbf E$ has a common fixed point. Furthermore, the restriction of $\mathbf E$ to the  set $Fix (\mathcal F)$  of common fixed points of $\mathcal F$ is a one-local retract of $\mathbf E$. 
\end{theorem}

\begin{proof}
For a subset $\mathcal {F'}$ of $\mathcal F$, let $Fix (\mathcal {F'})$ be the set of fixed points of $\mathcal {F'}$.

\begin{claim}\label{claim:finite} For every finite subset $\mathcal {F'}$ of $\mathcal F$, $\mathbf {E}_{\restriction Fix(\mathcal {F'})}$ is a one-local retract of $\mathbf E$.
 \end{claim}
\noindent {\bf Proof of  Claim \ref {claim:finite}.} Induction on $n:= \vert \mathcal {F'}\vert$. If $n=0$, there is no map, hence the set of fixed points is $E$, thus the conclusion holds. If $n=1$,  this is Theorem \ref{thm:main}. Let $n\geq 1$. Suppose that the property holds for every subset $\mathcal {F''}$ of $\mathcal F'$ such that $\vert\mathcal {F''}\vert<n$. Let $f\in \mathcal {F'}$  and $\mathcal F'':= \mathcal {F'} \setminus \{f\}$. From our inductive hypothesis, $\mathbf {E}_{\restriction Fix(\mathcal {F''})}$  is a one-local retract of $\mathbf E$. Thus, according to Lemma \ref{lem:compact},   $\mathbf {E}_{\restriction Fix(\mathcal {F''})}$ has a compact and normal structure. Now since $f$ commutes with every member $g$ of $\mathcal {F''}$, $f$ preserves $Fix(\mathcal {F''})$ (indeed, if  $u \in Fix(\mathcal {F''})$,  we have $g(f(u))=f(g(u))=f(u)$, that is $f(u) \in Fix(\mathcal {F''})$. Thus $f$ induces  an endomorphism $f''$ of $\mathbf {E}_{\restriction Fix(\mathcal {F''})}$.    According to Theorem \ref{thm:main},  the restriction of $\mathbf E_{\restriction Fix(\mathcal {F''}) }$ to $Fix (f'')$, that is $\mathbf E_{\restriction Fix(\mathcal {F'})}$,  is a one-local retract of $\mathbf E_{\restriction Fix(\mathcal {F''})}$. Since the notion of one-local retract is transitive  it follows that $\mathbf E_{\restriction Fix(\mathcal {F'})}$ is a one-local retract of $\mathbf E$. \hfill $\Box$

 Let $\mathcal P:= \{Fix (\mathcal {F''}): \vert \mathcal {F''}\vert < \aleph_{0} \}$ and $P:= \bigcap \mathcal P$. According to Theorem \ref{thm:best},  $\mathbf E_{\restriction  P}$ is a one-local retract of $\mathbf E$. Since $P= Fix(\mathcal {F})$ the conclusion follows.

\end{proof}
\section{Proof of Theorem \ref{thm:best}}

We recall the following basic fact about ordered sets (see \cite {cohn} Proposition 5.9 p 33). 
\begin{lemma} \label{lem:down}Every down directed subset of a partially ordered set $P$ has an infimum iff every totally ordered subset of $P$ has an infimum. 
\end{lemma}
  
Let $\mathbf E:= (E, \mathcal E)$  which is  is involutive, reflexiveand has a compact and normal structure. Let $P$ be the set, ordered by inclusion, of   subsets $A$ of $E$ such that $\mathbf E_{\restriction A}$ is a one-local retract of $\mathbf E$. 

We will prove the following: 

\begin{lemma}\label{lem:intersection} The set $P$  is closed under intersection of every chain of its members. 
\end{lemma}
We claim that with  the help of 
 Lemma \ref{lem:down} it follows that  $P$ is closed under intersection of every down directed family of its members. This statement  is Theorem \ref{thm:best}. 
 
 Indeed, observe first that from Lemma \ref{lem:intersection} it follows that for every subset $X$ of $E$, the set $P_X:= \{A\in P: X\subseteq A\}$ is closed under intersection of every chain of its members (if $\mathcal C$ is such a chain then $C:=\bigcap \mathcal C\in P$ by Lemma \ref{lem:intersection}, and trivially $X\subseteq C$, hence $C\in P_X$). Next, let $\mathcal A$ be a down directed family of members of $P$ and let $X:=\bigcap \mathcal A$. Then $\mathcal A \subseteq P_X$. Since $P_X$ is closed under intersection of every chain of its members, Lemma  \ref{lem:down} ensures that $\mathcal A$ has an infimum in $P_X$. This infimum must be  $X$.

In order to prove Lemma \ref{lem:intersection} we prove the following 
\begin{lemma}\label{lem:descending-intersection}Let $\mathbf E:= (E, \mathcal E)$ be a reflexive and involutive binary relational system with a compact and normal structure; let $\kappa$ be a cardinal. For every ordinal $\alpha$, $\alpha< \kappa$ let $B_{\alpha}, E_{\alpha}$ be subsets  of $E$ such that:
\begin{enumerate}
\item $B_{\alpha} \supseteq B_{\alpha+1}$ and $E_{\alpha} \supseteq E_{\alpha+1}$ for every $\alpha<\kappa$; 
\item $\bigcap_{\gamma<\alpha} B_{\gamma}=B_{\alpha}$ and $\bigcap_{\gamma<\alpha} E_{\gamma}=E_{\alpha}$ for every limit $\alpha< \kappa$; 
\item $\mathbf E_{\alpha}:=\mathbf E_{\restriction  E_{\alpha}}$ is a one-local retract of $\mathbf E$ and $B_{\alpha}$ is a nonempty intersection of balls of $\mathbf  E_{\alpha}$. 
\end{enumerate} 

Then $B_{\kappa}:=\bigcap_{\alpha< \kappa} B_{\alpha}\not = \emptyset$. 
\end{lemma}

Before proving the lemma, let us deduce Lemma \ref{lem:intersection} from it. 
We argue by induction on the size of totally ordered families of one-local retracts of $\mathbf E$. First we may suppose that $E$ has more than one element; next, we may suppose that these families are dually well ordered by induction. Thus,   given an infinite cardinal $\kappa$, let $(\mathbf E_{\restriction E_{\alpha}})_{\alpha<\kappa}$ be a descending sequence of one-local retracts  of $\mathbf E$. From the induction hypothesis, we may suppose that  the restriction of $\E$ to $E'_{\alpha}:= \bigcap_{\gamma<\alpha}E_{\gamma}$ is a one-local retract of $\mathbf E$ for each limit ordinal $\alpha< \kappa$. Hence, we may suppose that  $E_{\alpha}:= \bigcap_{\gamma<\alpha}E_{\gamma}$ for each limit ordinal $\alpha< \kappa$. 
Since $\E_{\alpha}$ is a one-local retract of $\mathbf E$ and  $\mathbf E$ has a normal structure, $\E_{\alpha}$ has a normal structure (Lemma \ref{lem:compact}).  Hence, either $E_{\alpha}$ is a singleton, say $x_{\alpha}$, or $r_{\mathbf E_{\alpha}}(E_{\alpha}) \setminus \delta_{\mathbf E_{\alpha}} (E_{\alpha})\not =\emptyset$. In both cases, $E_{\alpha}$ is a ball of $\mathbf E_{\alpha}$ (since $\mathbf E$ is reflexive, $(x_{\alpha}, x_{\alpha}) \in r$ for any $r\in \mathcal E$, hence  the first case, $E_{\alpha}= B_{\mathbf E_{\alpha}}(x_{\alpha}, r_{\restriction {E_{\alpha}}})$,  whereas in second case,  $E_{\alpha}\subseteq B_{\mathbf E_{\alpha}}(x, r)$ for some $x\in E_{\alpha}, r \in r_{\mathbf E_{\alpha}}(E_{\alpha}) \setminus \delta_{\mathbf E_{\alpha}} (E_{\alpha})$.  Hence,  Lemma \ref{lem:descending-intersection} applies with $B_{\alpha}= E_{\alpha}$ and gives that $E_{\kappa}$ is nonempty. Let us prove that $\E_{\kappa}:= \E_{\restriction E_{\kappa}}$ is a one-local retract of $\mathbf E$. We apply Lemma \ref{lem:one-local-retract}. Let $(B_{\mathbf E}(x_i, r_i ))_{i\in I}, x_i \in E_{\kappa}$, $r_i\in \mathcal E$ be a family of balls  such that   the intersection  is nonempty. Since $\E_{\alpha}$ is a one-local retract of $\mathbf E$, the intersection $B_{\alpha}:= E_{\alpha}\bigcap \bigcap_{i\in I}B_{\mathbf E}(x_i, r_i)$ is nonempty for every $\alpha<\kappa$. Now, Lemma \ref{lem:descending-intersection} applied to the sequence $(E_{\alpha}, B_{\alpha})_{\alpha<\kappa}$ tells us that $B_{\kappa}:=E_{\kappa}\cap \cap_{i\in I} B_{\mathbf E}(x_i, r_i)$ is nonempty. According to Lemma \ref{lem:one-local-retract}, $\mathbf {E}_{\restriction B_{\kappa}}$ is a one-local retract of $\mathbf E$.   
\subsection{Proof of Lemma \ref{lem:descending-intersection}}

Let $\mathfrak A$ be the collection of all descending sequences $\mathcal A:= (A_{\alpha})_{\alpha < \kappa}$ such  that each $A_{\alpha}$ is a nonempty intersection of balls of $\mathbf E_{\restriction E_{\alpha}}$ included into $B_{\alpha}$.  Set $\mathbf E_{\alpha}:= \mathbf E_{\restriction E_{\alpha}}$ and  $\mathfrak B:=\Pi_{\alpha<\kappa}\hat B^{\ast}_{\mathbf E_{\alpha}}$. 

The sequence $\mathcal B:= (B_{\alpha})_{\alpha<\kappa}$ belongs to  $\mathfrak A$ and   $\mathfrak A$ is included into $\mathfrak B$. The set $\mathfrak B$ is ordered as follows:

\begin{equation}
(A'_{\alpha})_{\alpha< \kappa}\leq (A''_{\alpha})_{\alpha<\kappa}\;  \text {if} \;  A'_{\alpha}\subseteq  A''_{\alpha}\;  \text{for every} \;  \alpha<\kappa. 
\end{equation}

Since $\E_{\alpha}$ is a one-local retract of $\mathbf{E}$, $\mathbf E_{\alpha}$ has a normal and compact structure (Lemma \ref{lem:compact}). Since it has a compact structure, every descending sequence in $\hat B^{\ast}_{\mathbf E_{\alpha}}$ has an infimum (Lemma \ref{lem:infimum}). Thus,  there is a minimal sequence $\mathcal A:=(A_{\alpha})_{\alpha< \kappa}$ with $\mathcal A\leq \mathcal B$. 

We prove first that the sequence of  $r_{\mathbf E}(A_{\alpha})$ is constant (see (\ref{sublem:item5}) of Sublemma \ref{sublem2}). Let $\mathbf r$ be the common value. We prove next that $\delta_{\mathbf E}(A_{\alpha})=\mathbf r$ (see (\ref{sublem:item6}) of Sublemma \ref{sublem2}). Since $\mathbf E$ has a normal structure, we deduce that each $A_{\alpha}$ is a singleton. Since $\mathcal A$ is decreasing, $A_{\kappa}:= \bigcap_{\alpha< \kappa}A_{\alpha}$ is a singleton too. Hence, $B_{\kappa}\not = \emptyset$. This is the conclusion of  the lemma. 

The key argument for the proof of Sublemma \ref{sublem2} is the following. 

\begin{sublemma}\label{lem:descendingball} Let $\alpha <\kappa$  and $A_{\alpha}\subseteq B_{\mathbf E}(x, r)$, with $r\in \mathcal E$ and $x\in E_{\alpha}$. Then $A_{\beta}\subseteq B_{\mathbf E}(x, r)$ for each $\beta<\kappa$. 
\end{sublemma}
\begin{proof} Set $B:=  B_{\mathbf E} (x,r)$. For $\xi< \kappa$ set $A'_{\xi}:= A_{\xi}\cap B$ if $\xi\leq\alpha$ and $A'_{\xi}= A_{\xi}$ otherwise. The familly $\mathcal A':= (A'_{\xi})_{\xi<\kappa}$  belongs to $\mathfrak A$ and satisfies $\mathcal A'\leq \mathcal A$. Since $\mathcal A$ is minimal,  we get $\mathcal A'=\mathcal A$, thus $A_{\xi}= A_{\xi}  \cap B$ for $\xi\leq \alpha$ that is $A_\xi\subseteq B$; since $A_{\xi}\subseteq A_{\alpha} \subseteq B$ for $\xi\geq \alpha$ it follows that $A_\xi\subseteq B$. 
\end{proof}

Let $\alpha< \kappa$. From the hypotheses of the lemma, there is a family  $\mathcal B':= (B_{\mathbf E_{\restriction E_{\alpha}}} (x'_i, r'_i))_{i\in I}$, with $x'_i\in E_{\alpha}$, $r'_i \in \mathcal E_{\restriction  E_{\alpha}}$ such that  $A_{\alpha}=\bigcap  \mathcal  B'$.  For each $i\in I$, let $r_i$ such that $r_{i \restriction E_{\alpha}}= r'_i$. Let $\mathcal B:=  (B_{\mathbf E_{\restriction E_{\alpha}}} (x'_i, r'_i))_{i\in I}$, $B=\bigcap  \mathcal B$. Then $A_{\alpha}= B\cap E_{\alpha}$.

From the sublemma above we deduce:

\begin{corollary} \label{coro:key}
Let $\alpha< \kappa$. Then:
\begin{enumerate}[{(a)}]
\item $A_{\beta}\subseteq B\cap E_{\beta}$ for every $\beta< \alpha$; 
\item $A_{\alpha}= \bigcap_{\beta<\alpha} A_{\beta}$ if  $\alpha$ is a limit ordinal; 
\item $r_{\mathbf E} (A_{\alpha})\subseteq  r_{\mathbf E} (A_{\beta})$ for every $\beta< \alpha$; 
\item $r_{\mathbf E} (A_{\beta})\subseteq  r_{\mathbf E} (A_{\alpha})$ for every $\beta< \alpha$. 
\end{enumerate}
\end{corollary}

\begin{proof}
$(a)$.  Follows directly from Sublemma \ref{lem:descendingball}. Indeed, we have $A_{\alpha} \subseteq B_{\mathbf E} (x'_i, r_i)$ hence from Sublemma  \ref{lem:descendingball}, $A_{\beta} \subseteq  B_{\mathbf E}(x, r)$. This yields $A_{\beta}\subseteq \bigcap \mathcal B=B$.

$(b)$. From $\bigcap_{\gamma<\alpha} E_{\gamma}=E_{\alpha}$ and $(a)$ we get 

$$A_{\alpha}=B \cap E_{\alpha}= \bigcap_{\beta<\alpha}B\cap E_{\beta}\supseteq \bigcap_{\beta<\alpha} A_{\beta}.$$

This implies $A_{\alpha}\supseteq \bigcap_{\beta<\alpha} A_{\beta}$. 

Since $\mathcal A$ is decreasing, we have $A_{\alpha}\subseteq \bigcap_{\beta<\alpha} A_{\beta}$. Hence, $A_{\alpha}=\bigcap_{\beta<\alpha} A_{\beta}$.

$(c)$.  Let $r\in  r_{\mathbf E} (A_{\alpha})$. Then $A_{\alpha}\subseteq B_{\mathbf E}(x, r)$ for some $x\in A_{\alpha}$. According to Sublemma \ref{lem:descendingball} we have $A_{\beta}\subseteq B_{\mathbf E}(x,r)$. Since $A_{\alpha}\subseteq A_{\beta}$, $x\in A_{\beta}$,  hence   $r\in r_{\mathbf E} (A_{\beta})$. 

$(d)$.  Let $r\in  r_{\mathbf E} (A_{\beta})$ and $x\in A_{\beta}$ such that  $A_{\beta}\subseteq B_{\mathbf E}(x, r)$.
From $(a)$ of Sublemma \ref{lem:descendingball}, we have $A_{\beta} \subseteq B$ thus $x\in B\cap \bigcap_{u\in A_{\alpha}}B(u, r^{-1})$.  Since $\mathbf{E}_{\restriction  E_{\alpha}}$ is a one-local retract, there is some $y\in B\cap \bigcap _{u\in A_{\alpha}}B_{\mathbf E}(u, r^{-1})\cap E_{\alpha}$. This means $A_{\alpha}\subseteq B_{\mathbf E}(y,r)$ which in turns implies $r\in r_{\mathbf E}(A_{\alpha})$. 

\end{proof}

\begin{sublemma} \label{sublem2}

\begin{enumerate} [{(a)}]
\item $r_{\mathbf E} (A_{\alpha})$ is independent of $\alpha$; \label{sublem:item5}
\item $\delta_{\mathbf E}(A_{\alpha})= r_{\mathbf  E}(A_{\alpha})$ for every $\alpha<\kappa$. \label{sublem:item6}
\end{enumerate}
\end{sublemma}
\begin{proof}

$(a)$.  Follows from $(c)$ and $(d)$ of Corollary \ref{coro:key}. 

$(b)$. Let $\mathbf r$ be the common value of all $r_{\mathbf  E} (A_{\alpha})$. Let $r\in \mathbf r$. Set $C_{r}(A_{\alpha}):= \{x \in E_{\alpha} : A_{\alpha}\subseteq B_{\mathbf E}(x, r)\}$,  $A^{r}_{\alpha}:= A_{\alpha} \cap C_{r}(A_{\alpha})$ and $\mathcal A^{r} := (A^{r}_{\alpha})_{\alpha<\kappa}$. 

\begin{claim}\label{claim:stationnary}
\begin{enumerate}[{(1)}]
\item $A^r_{\alpha}$ is a nonempty intersection of balls of $\mathbf E_{\restriction E_{\alpha}}$;
\item $A^r_{\alpha} \subseteq A_{\alpha}$; 
\item $A^r_{\beta} \supseteq A^r_{\alpha}$ for $\beta <\alpha$.
\end{enumerate}
\end{claim}
\noindent{\bf Proof of Claim \ref{claim:stationnary}.}
$(1)$. Since $r\in r_{\mathbf E}(A_{\alpha})$, $A_{\alpha}\subseteq B_{\mathbf E}(x, r)$ for some $x \in A_{\alpha}$, hence $x\in C_{r}(A_{\alpha})$ proving that $A^{r}_{\alpha}$ is nonempty. Since $\mathbf E$ is involutive, $r^{-1} \in \mathcal E$, thus from  $(iii)$ of Proposition \ref{prop:center}, $C_{r}(A_{\alpha})$ is an intersection of balls of $\mathbf E_{\restriction E_{\alpha}}$  with centers in $A_{\alpha}$. Hence, $A^r_{\alpha}$ is a nonempty intersection of balls of $\mathbf E_{\restriction E_{\alpha}}$.  

$(2)$ Obvious. 

$(3)$. Let $\beta<\alpha$.  By construction of $\mathcal A$, we have $A_{\beta} \supseteq A_{\alpha}$. Let $x\in A^r_{\alpha}$. By definition, we have $A_{\alpha} \subseteq B_{\mathbf E}(x,r)$. From Lemma \ref{lem:descendingball}, we have $A_{\beta} \subseteq B_{\mathbf E}(x,r)$. It follows that $x\in C_{r}(A_{\beta})$. Since $x\in A_{\beta}$,  $x\in A_{\beta}^r$. This proves that $(3)$ holds. \hfill $\Box$

From Claim \ref{claim:stationnary} and the minimality of $\mathcal A$ we obtain $\mathcal A^{r}=\mathcal A$.
From this it follows that $A_{\alpha}\subseteq C_{r}(A_{\alpha})$. Since this inclusion holds for every $r\in r_{\mathbf E} (A_{\alpha})$ we get $\delta_{\mathbf E}(A_{\alpha})= r_{\mathbf E} (A_{\alpha})$. This proves that $(\ref{sublem:item6})$ holds. This ends the proof of Sublemma  \ref{sublem2}. \end{proof}

\section{Illustrations}
 
\subsection{Preservation}\label{subsection:preserve}
Let $E$ be a set.   For $n\in \N^*:= \N\setminus
\{0\}$, a map $f: E^n
\rightarrow E$ is an
$n$-{\it ary operation} on $E$, whereas a subset 
$\rho \subseteq E^{n}$ is an $n$-{ary relation} on $E$.
Denote by
$\mathcal O^{(n)}$ (resp.$\mathcal R^{n}$) the set of
$n$-ary operations (resp. relations)  on $E$ and set 
$\mathcal O:=\bigcup \{\mathcal O^{(n)}: n\in N^* \}$ (resp 
$\mathcal R:= \bigcup \{\mathcal R^{(n)}:
n\in N^* \}$). For  $n, i\in N^*$  with $i\leq n$, define the   $i^{th}$
$n$-{\it ary projection}
$e^{n}_{i}$ by setting $e^{n}_{i}(x_{1},\dots,x_{n}):= x_{i}$ for all  $x_{1},\dots,x_{n}\in E$ and set 
$\mathcal P:= \{e^{n}_{i}: i,n \in \N^*\}$. An operation $f\in  \mathcal O$ is {\it constant} if it takes a single value, it 
 is\ {\it idempotent} provided
$f(x,\dots,x)= x$ for all $x\in E$. We denote by  $\mathcal C$ (resp. $\mathcal I$) the set of  constant, (resp. idempotent)  operations on $E$. 

Let $m,n\in \N^*$,
$f \in \mathcal O^{(m)}$ and $\rho \in \mathcal R^{(n)}$.    Then $f$ {\it preserves}
$\rho$
 if: 
\begin{equation}
\small {(x_{1,1}, \dots, x_{1,n})\in
\rho, \dots,
(x_{m,1}, \dots, x_{m,n})\in \rho \Longrightarrow (f(x_{1,1}, \dots, x_{m,1}),\dots,f(x_{1,n}, \dots,
x_{m,n}))\in \rho}
\end{equation}
for every $m\times n$ matrix
$X:= (x_{i,j})_{i=1,\ldots,m \atop {j=1, \ldots ,n}}$ of elements of $E$. 

If $\rho$ is binary and $f$ is unary, then $f$ preserves $\rho$ means:

\begin{equation}
(x, y)\in \rho
 \Longrightarrow (f(x),f(y))\in \rho 
\end{equation}

for all $x, y \in E$. 

If $\mathcal F$ is a set of operations on $E$, let $Inv(\mathcal F)$, resp. $Inv_n(\mathcal F)$ be the set of relations, resp. $n$-ary relations,  preserved by all $f\in \mathcal F$. Dually, if $\mathcal R$ is a set of relations on $E$, let $Pol(\mathcal R)$, resp. $Pol_n(\mathcal R)$,  be the set of operations, resp. $n$-ary operations,  which prreserve all $\rho\in \mathcal R$. The operators Inv and Pol define a Galois correspondance. The study of this correspondence is the theory of clones \cite{lau}. 

\subsection{Toward generalized metric spaces} We restrict our attention to the  case of unary operations and binary relations. We recall that if $\rho$ and $\tau$ are two binary relations on the same set $E$, then their composition $\rho\circ\tau$ is the binary relation made of pairs $(x,y)$ such that $(x,z)\in \tau$ and $(z,y)\in \rho$. It is customary to denote it $\tau\cdot \rho$. 

The set $Inv_2(\mathcal F)$ of binary relations on $E$ preserved by all $f$ belonging to a set $\mathcal F$ of self maps has some very simple  properties that we state below (the proofs are left to the reader).  For the construction of many more properties by means of primitive positive formulas, see \cite{snow}. 

\begin{lemma}\label{lem:monoid} Let $\mathcal F$ be a set of unary operations on a set $E$. 
Then the set  $\mathcal R:=Inv_2(\mathcal F)$ of binary relations on $E$ preserved by all $f\in \mathcal F$ satisfies the following properties:

\begin{enumerate}[{(a)}]
\item $\Delta_E\in \mathcal R$; 
\item $\mathcal R$ is closed under arbitrary intersections; in particular $E\times E\in \mathcal R$; 
\item $\mathcal R$ is closed under arbitrary unions; 
\item If $\rho, \tau \in \mathcal R$ then $\rho\circ \tau \in \mathcal R$;
\item If $\rho\in \mathcal R$ then $\rho^{-1}\in \mathcal R$.

\end{enumerate} 
\end{lemma}

Let $\mathcal R$ be a set of binary relations on  a set $E$ satisfying  items $(a)$, $(b)$, $(d)$ and $(e)$ of the above lemma (we do not require $(c)$). To make things more transparent, denote by $0$ the set $\Delta_E$, set $\rho\oplus \tau:= \rho\cdot \tau$. Then $\mathcal R$ becomes  a monoid.  Set $\overline {\rho}:= \rho^{-1}$, this defines an involution on $\mathcal R$ which reverses the monoid operation. With this involution $\mathcal R$ is an \emph{involutive monoid}. With the inclusion order, that we denote  $\leq$,  this involutive   monoid is an \emph{involutive complete ordered monoid}.

With these definitions, we have immediately:

\begin{lemma} \label{lem:distance}Let $\mathcal R$ be an involutive  complete ordered monoid of the set of binary relations on $E$ and let 
$d$ be the map from $E\times E$ into $\mathcal R$ defined by  $$d(x,y):= \bigcap \{\rho\in \mathcal R: (x,y)\in \rho\}.$$
Then, the following properties hold:

\begin{enumerate}[{(i)}]
\item $d(x,y)  \leq 0$ iff $x=y$;
\item $d(x,y) \leq d(x,z)\oplus d(z, y)$; 
\item $\overline{d(y,x)}=d(x,y)$. 
\end{enumerate}
\end{lemma}

In 
\cite{pouzet-rosenberg}, a set $E$ equipped with a map $d$ from $E\times E$ into an involutive ordered monoid $V$ and which satisfies   properties $(i), (ii), (iii)$     stated in Lemma \ref {lem:distance} is called a \emph{$V$-distance}, and the pair $(E,d)$ a \emph{$V$-metric space}. This lemma could justify that  we write  $d(x,y) \leq \rho$ the fact that  a pair $(x,y)$ belongs to a binary relation $\rho$ on the  set $E$   and then uses notions borrowed to the theory of metric spaces. 

{\bf N.B.} From now on, we suppose that the neutral element of the monoid $V$ is the least element of $V$ for the ordering. In   \cite{deza-deza} (cf. p.82) the corresponding $V$-metric spaces are called  \emph{generalized distance space} and the maps $d$  are called  \emph{generalized metric}.  

If  $(E, d)$ is a $V$-metric space and $A$ a subset of $E$, the restriction of $d$ to $A\times A$, denoted by $d_{\restriction A}$ is a $V$-distance 
and $(A, d_{\restriction A})$ is a \emph{restriction} of $(E, d)$. As in the case of ordinary metric spaces, if  $(E,d)$ and $\left( E^{\prime }, d^{\prime }\right) $ are two
$V-$metric spaces,  a map $f:E\longrightarrow E^{\prime }$ is a
\emph{non-expansive map} (or a \emph{contraction}) from $(E$, $d)$ to $\left( E^{\prime },d^{\prime
}\right) $ provided that $d^{\prime }(f(x),f(y))\leq d(x,y)$ holds for all $%
x,y\in E$ (and the map $f$ is an \textit{isometry} if $d^{\prime
}(f(x),f(y))=d(x,y)$ for all $x,y\in E$). 
The space $(E,d)$ is a \textit{retract} of $(E^{\prime},d')$, if there are two non-expansive maps $f:
E\longrightarrow E^{\prime }$ and $g:E^{\prime
}\longrightarrow E$ such that $g\circ f=id_{E}$ (where $id_{E}$ is
the identity map on $E$). In this case,
$f$ is a  \textit{coretraction} and $g$ a \textit{%
retraction}. If $E$ is a subspace of $E^{\prime }$, then clearly $E$ is a retract of
$E^{\prime }$ if there is a non-expansive map from $E^{\prime }$ to $E$ such $%
g(x)=x $ for all $x$ $\in E.$ We can easily see that every coretraction is
an isometry. We say that $(A, d_{\restriction A})$  is a \emph{one-local retract} if it is a retract of $(A\cup \{x\},  d_{\restriction A\cup \{x\}})$ (via the identity map) for every $x\in E$.

Let  $(E, d)$ be a $V$-metric space; for  $x \in E$ and $v\in V$, the set $B(x,v):= \{y\in E: d(x,y)\leq v\}$ is  a \emph{ball}. One can define diameter and radius like in ordinary metric spaces, but for avoiding  a problem with the existence of joins and meets,  we suppose that $V$ is a complete lattice. The \emph{diameter}  $\delta(A)$ of a subset $A$ of $E$ is $\bigvee \{d(x,y): x,y\in A\}$, while the \emph{radius} $r(A)$ is $\bigwedge \{v\in V: A\subseteq B(x,v)\; \text{for some}\; x\in A\}$. A subset $A$ of $E$ is \emph{equally centered} if $\delta (A)=r(A)$.  Following  Penot,  who defined the notions for ordinary metric spaces, a metric space $(E,d)$ has a \emph{compact structure} if the collection of balls has the finite intersection property and it has  a \emph{normal structure} if for every   intersection of balls $A$, either $\delta(A)=0$ or $r(A)<\delta(A)$; this condition amounts to the fact that  the only equally centered intersections of balls are singletons. 

The correspondence between the  notions defined for metric spaces and for binary relational systems is given in the lemma below. 

\begin{lemma}\label{lem:metric-relation}
For $v\in V$,  set $\delta_v:= \{(x,y): d(x,y)\leq v\}$ and   $\mathbf E:= (E, \{\delta_v: v\in V\})$. Then $\mathbf E$ is reflexive and involutive. Furthermore:
\begin{enumerate}[{(a)}]
\item A self map $f$ on $E$ is nonexpansive iff this is an endomorphism of $\mathbf E$.

\item $(E,d)$ has a compact structure iff $\mathbf E$ has a compact  structure.

\item For every subset $A$ of $E$, $(A, d_{\restriction A})$  is a one-local retract of $(E,d)$ iff $\mathbf {E}_{\restriction A}$ is a one-local retract of $\mathbf E$.

\item  For every subset  $A$ of $E$, $\delta(A)$ is the least  element of the set of $v\in V$ such that $A\subseteq \delta_{v}$; equivalently $\delta_{\mathbf E}(A) = \{\delta_v: \delta (A)\leq v\}$. Also, $r(A)= \bigwedge \{v\in V: \delta_v\in r_{\mathbf{E}}(A)\}$. 

\item A subset $A$ of $E$ is equally centered w.r.t. the space $(E,d)$  iff it is equally centered w.r.t. the binary relational system $\mathbf E$.

\end{enumerate}
\end{lemma}
\begin{proof}
The first three items are obvious. 

Item $(d)$. Let $r:= \delta(A)$. By definition, $r= \bigvee \{d(x,y): (x,y)\in A^2\}$. In particular, $A\subseteq \delta_r$. Let $v$ such that $A\subseteq \delta_v$; this means $d(x,y)\leq v$ for every $(x,y)\in A^2$, hence $r\leq v$. This proves that $\delta(A)= \Min \{v\in V: \delta_v\in \delta_{\mathbf{E}}(A)\}$. The verification of the other assertions is immediate. 

Item $(e)$.  By Item $(d)$, $r(A):= \bigwedge\{v\in V: \delta_v\in r_{\mathbf{E}}(A)\}$ and $\delta(A):= \Min  \{v\in V: \delta_v\in \delta_{\mathbf{E}}(A)\}$. If $r_{\mathbf E}(A)= \delta_{\mathbf E}(A)$, this implies immediately  $r(A)= \delta(A)$. Conversely, suppose that $r(A)= \delta(A)$. In this case $A\not = \emptyset$, hence $\delta_{\mathbf E}(A)\subseteq r_{\mathbf E}(A)$. If $\delta_{\mathbf E}(A)\subset r_{\mathbf E}(A)$ then since $\delta (A)=\Min  \{v\in V: \delta_v\in \delta_{\mathbf{E}}(A)\}$ and  $r(A)=\bigwedge\{v\in V: \delta_v\in r_{\mathbf{E}}(A)\}$ it follows that $r(A)<\delta (A)$, a contradiction. \end{proof}
With this lemma in  hand, Theorem \ref {thm:cor} becomes:

\begin{theorem} \label{thm:cor2}
If   a generalized metric space $(E,d)$   has a   has a compact and normal structure then every commuting family $\mathcal F$ of non-expansive self maps has a common fixed point. Furthermore, the restriction of $(E,d)$ to the  set $Fix (\mathcal F)$  of common fixed points of $\mathcal F$ is a one-local retract of $(E,d)$. 
\end{theorem}

The fact that a space has a compact structure is an infinistic property (any finite metric spaces enjoys it). 
A description of generalized metric spaces with a compact and normal structure eludes us.  In the next subsection we describe a large class of generalized metric spaces with a compact and normal structure.

\subsection{Hyperconvexity}

We say that a generalized metric space $(E,d)$ is \emph{hyperconvex} if for every family of balls $B(x_i, r_i)$, $i\in I$,  with $x_i\in E, r_i\in V$,  the intersection $\bigcap_{i\in I}B(x_i, r_i)$ is nonempty provided that $d(x_i, x_j)\leq r_i\oplus \overline r_j$ for all $i,j\in I$. This property amounts to the fact that the collection of balls of $(E,d)$ has the \emph{$2$-Helly property} (that is  an intersection of balls is nonempty provided that these balls intersect pairwise)  and the following \emph{convexity property}: 

\begin{equation}
 \mbox {Any two balls}\;   B(x,r),  B(y,s) \; \mbox{intersect if and only if}\; d(x,y)\leq r\oplus \overline s. 
\end{equation}

An element $v\in V$ is \emph{self-dual} if $\overline v=v$, it is   \emph{accessible} if there is some $r\in V$ with $v\not \leq r$ and $v \leq r\oplus \overline r$ and \emph{inaccessible} otherwise. Clearly, $0$ is inacessible;   every inaccessible element $v$ is self-dual (otherwise, $\overline v$ is incomparable to $v$ and we may choose $r:= \overline v$). We say that a space $(E, d)$ is \emph{bounded} if  $0$ is the only inaccessible element below $\delta(E)$.

\begin{lemma}
Let $A$ be an intersection of balls  of $E$. If $\delta(A)$ inacessible  then  $A$ is equally centered; the  converse holds if $(E, d)$ is hyperconvex. 
\end{lemma}
\begin{proof}
Suppose that $v:=\delta(A)$ is inaccessible. According to $(d)$ of Lemma \ref{lem:metric-relation}, $r(A)=\bigwedge r_{\mathbf E}(A)$. Let $r\in  r_{\mathbf E}(A)$. Then there is some $x\in A$ such that  $A \subseteq B(x, r)$. This yields $d(a,b)\leq d(a,x)\oplus d(x,b)\leq \overline r\oplus r$ for every $a,b\in A$. Thus $v\leq \overline r\oplus r$. Since $v$ is inacessible,  $v\leq r$, hence $v\leq r(A)$. Thus $v= r(A)$. Suppose that $A$ is equally centered. Let $r$ such that $v\leq r\oplus \overline r$. The  balls $B(x, r)$ ($x\in A$) pairwise intersect and intersect with each of the balls whose $A$ is an intersection; since $(E, d)$ is hyperconvex, these balls have a nonempty intersection. Any member $a$ of  this intersection is in $A$ and verifies $A \subseteq B(a, \overline r)$,  hence $\overline r\in r_{\mathbf E}(A)$. Since $A$ is equally centered $r(A)=v$. Hence, $v\leq \overline r$. Since $v$ is self-dual,  $v\leq r$. Thus $v$ is inaccessible.   
\end{proof}

This lemma with the fact that the $2$-Helly property implies that the collection of balls has the finite intersection property yields:
\begin{corollary}\label{cor:compact+normal} If a generalized   metric space $(E,d)$  is bounded and hyperconvex then it has a compact and normal structure. 
\end{corollary}

From Theorem \ref{thm:cor2},  we obtain:
\begin{theorem} \label{thm:cor3}
If a generalized   metric space $(E,d)$  is bounded and hyperconvex  then every commuting family  of non-expansive self maps has a common fixed point. 
\end{theorem}

Hyperconvex spaces have a simple characterization provided  that the set $V$ of values of the distances satisfies the following distributivity condition:

\begin{equation}\label{distribinfty}
\bigwedge  _{\alpha \in A, \beta \in B} u_\alpha \cdot  v_\beta =
\bigwedge _{\alpha \in A} u_\alpha  \cdot  \bigwedge _{\beta \in B} v_\beta
\end{equation}
for all $u_\alpha \in V$ $(\alpha \in A)$ and
$v_{\beta} \in V$ $(\beta \in B)$.

In this case,  we say  that $V$ is an  \emph{involutive Heyting algebra}  or, better, an \emph{involutive op-quantale} (see \cite {rosenthal} about quantales). 

On an involutive  Heyting algebra $V$, we may define a $V$-distance. This fact relies on the classical notion of residuation. Let $v\in V$. Given $\beta \in V$, the sets
 $\{r \in V: v \leq r \oplus \beta\}$ and
$\{r \in V: v \leq  \beta \oplus r \}$ have least elements, that we
denote
respectively  by $\lceil v\oplus -\beta\rceil$ and $\lceil -\beta\oplus  v  \rceil$ and call the \emph{right} and \emph{left  quotient} of $v$ by $\beta$ (note that  $\overline {\lceil -\beta\oplus v \rceil} =
\lceil \bar v\oplus {-{\bar\beta}} \rceil$). It follows that for all
$p, q \in V$, the set
\begin{equation}\label{distance}
D(p,q):=\{r \in V : p\leq q \oplus \bar r\;\;{\rm and}\; \; q\leq
p\oplus r\} 
\end{equation} 
has a least
element.
This last element is $\lceil \bar p\oplus -\bar q \rceil \vee
\lceil -p\oplus q \rceil$, we  denote it by $d_V(p,q)$. 

As shown in \cite{jawhari-al}, the
map $(p,q) \longrightarrow d_{V}(p,q)$
is a $ V-$distance.

Let $\left( (E_{i}, d_{i})\right) _{i\in I}$ be a family of $V$-metric spaces. The \emph{direct product}  $\underset{i\in I}{%
\prod }\left( E_{i}, d_{i}\right) $, is the metric space $(E,d) $ where $E$ is the cartesian product  $%
\underset{i\in I}{\prod }E_{i}$ and $d$ is the  ''sup'' (or
$\ell ^{\infty }$) distance  defined
by $d\left(
\left( x_{i}\right) _{i\in I},\left( y_{i}\right) _{i\in I}\right) =%
\underset{i\in I}{\bigvee }d_{i}(x_{i},$ $y_{i})$.
We recall the following result of \cite{jawhari-al}. 

\begin{theorem}\label{thm:hyper1} $(V,d_{V})$ is a hyperconvex  $V$-metric space
and every $V$-metric space embeds isometrically into a
power of $(V,d_{V})$. 
\end{theorem}

This result is due to the fact that for
every $V$-metric space $(E,d)$ and for all $x,y\in E$ the
following equality holds:
\begin{equation}\label{metricsup}
d\left( x,y\right)  = \underset{
z\in E}{\bigvee } d_{V} (d(z,x),d(z,y)).
\end{equation}

 A generalized metric space is an \textit{absolute retract} if it is a
retract of every isometric extension. The space $E$ is  \textit{%
injective} if for all $V$-metric space $E^{\prime }$ and $E'', $ each
non-expansive map $f:E^{\prime }\longrightarrow E$ and every isometry $g:E^{\prime
}\longrightarrow E''$ there is a non-expansive map $h:E''\longrightarrow E$ such that $%
h\circ g=f$.  

With this result follows the characterization given in \cite{jawhari-al}.  

\begin{theorem}\label{thm:hyper2}For  metric spaces  over an involutive  Heyting algebra $V$,
the notions of absolute retract, injective, hyperconvex and retract of a
power of $(V,d_{V})$ coincide.
\end{theorem}

Note that if $v$ is accessible in $V$  and $V$ is an involutive  Heyting algebra, then $v$ is accessible in the initial segment $\downarrow v$ of $V$ (indeed, if $v\leq r\oplus s$ then since by distributivity $(r\wedge v)\oplus (\overline r\wedge v)= (r \oplus \overline r)\wedge (r \oplus v)\wedge (v \oplus \overline r)\wedge (v \oplus v)$, we have $v\leq (r\wedge v)\oplus (\overline r\wedge v)$. 

\subsection{One-local retracts and hole-preserving maps.} 

Let $E$ and $E'$ be two $V$-metric spaces. If $f$ is a non-expansive map from $E$ into  $E'$, and $h$ is a map from  $E$ into $V$, the \emph{image} of $h$ is the map $h_f$ from $E'$ into $V$ defined by $h_f(y): \bigwedge \{h(x):= f(x)=y\}$ (in particular $h_f(y)= 1$ where $1$ is the largest element of $V$ for every $y$ not in the range of $f$.  A \emph{hole} of $E$ is any map $h:E \rightarrow V$ such that  the intersection of balls $B(x, h(x))$ of $E$ ($x\in E$) is empty. If $h$ is a hole of $E$, the map $f$ \emph{preserves} $h$ provided that $h_f$ is a hole of $E'$. The map $f$ is \emph{hole-preserving} if the image of every hole is a hole.   

Let $\mathcal B:= (B(x_i, r_i))_{i\in I}$ be a family of balls of $E$. For every $x\in E$,  set $V_{\mathcal B}(x)=\{r\in V: B(x_i,r_i)\subseteq B(x,r) \;   \text{for some}\; i\in I\}$ and  $h_{\mathcal B}(x):= \bigwedge V_{\mathcal B}(x)$.   We have:

\begin{equation}\label{eq:hole}
\bigcap \mathcal B= \bigcap_{x\in E} B(x, h_{\mathcal B}(x)). 
\end{equation}

\begin{proof} Let $z\in \bigcap\mathcal B$ and $x\in E$. We claim that $z\in B(x, h_{\mathcal B}(x))$. This amounts to $d(x,z)\leq h_{\mathcal B}(x)$; due to the definition of $h_{\mathcal B}(x)$, this amounts to $d(x,z)\leq r$ for every $r\in V$ such that $B(x_i,r_i)\subseteq B(x,r)$  for some $i\in I$. Let $r$ and $i\in I$ such that this property holds. Since $z\in \bigcap\mathcal B$, $z\in B(x_i,r_i)$ and since $B(x_i,r_i)\subseteq B(x,r)$, $z\in B(x,r)$,  that is $d(x,z)\leq r$,  as required.  For the converse,  let  $z\in \bigcap_{x\in E} B(x, h_{\mathcal B}(x))$ and $i\in I$. By definition of $h_{\mathcal B}$,  $h_{\mathcal B}(x_i)\leq r_i$; since $z\in B(x_i, h(x_i))$, $z\in B(x_i,r_i)$; since this property holds for every $i\in I$, $z\in \bigcap \mathcal B$. 

\end{proof}

A hole $h$ of $E$ is \emph{finite} if $\bigcap_{x\in F} B(x, h(x))=\emptyset $ for some finite subset $F$ of $E$, otherwise it is infinite. 

A poset is \emph{well-founded} if every nonempty subset contains some minimal element. We recall that if a lattice is well-founded, every element $x$ which is the infimum of some subset $X$ is the infimum de some finite subset. In general,  the order on the Heyting algebra  $V$ is not well-founded,  still there are interesting examples (see Subsection \ref{subsection:orderedset} and \ref{subsection:orientedgraphs}). 

The following lemma  relates holes and compactness of the collection of balls (it contains a correction of Proposition II-4.9. of \cite{jawhari-al}):
\begin{lemma}\label{lem:hole}
If a   generalized space $E$ has a compact structure then every  hole is finite; the converse holds if $V$ is well-founded. 
\end{lemma}
\begin{proof} Let $ h$ be  a hole. Then, by definition, $\bigcap_{x\in E} B(x, h(x))=\emptyset$. Since $E$ has a compact structure, $\bigcap_{x\in F} B(x, h(x))=\emptyset$ for some finite subset, hence $h$ is finite. Conversely, let  $\mathcal B:= (B(x_i, r_i))_{i\in I}$ be a family of balls of $E$ such that $\bigcap \mathcal B=\emptyset$. There are two ways of associating a finite hole to $\mathcal B$. We may define  $h:\rightarrow V$ be setting $h(x):= \bigwedge \{r_i: x_i=x\}$. We may also associate  $h_{\mathcal B}$. By Formula (\ref{eq:hole}), this   is a hole. These hole are  finite. We conclude by using $h_{\mathcal B}$. Let $F$ be some finite subset of $E$ such that $\bigcap_{x\in F} B(x, h_{\mathcal B}(x))=\emptyset$. Since $V$ is well-founded, for each $x\in E$, there is some finite subset $V_x$ of $V_{\mathcal B}(x)=\{r\in V: B(x_i,r_i)\subseteq B(x,r)  \text{for some}\; i\in I\}$ such that $\bigwedge V_{\mathcal B}(x)=\bigwedge V_x$. For  each  $x\in F$, there is a finite subset $I_x$ such that for each $r\in V_x$ there is some $i\in I_x$  such that $B(x_{i},r_{i})\subseteq B(x,r)$. Then $\bigcap_{i\in \bigcup_{x\in F}  I_x}B(x_i, r_i)=\emptyset$ proving that  the intersection of finitely many many members of $\mathcal B$ is nonempty. 
\end{proof}

\begin{lemma} \label{lem:hole-localretract}A non-expansive map $f$ from a $V$-metric space $E$ into a $V$-metric space  $E'$ is hole-preserving iff  $f$ is an isometry of $E$ onto its image and this image is a $1$-local retract of $E'$.
\end{lemma}
\begin{proof}Let $d$ and $d'$ be the distances of $E$ and $E'$. 
Suppose that $f$ is hole-preserving. We prove first that $f$ is an isometry. Let $a,b\in E$, $a':=f(a)$, $b':=f(b)$, $r:=d(a,b)$ and $r':= d'(a',b')$. Our aim is to prove that $r'=r$. Let $h: E\rightarrow V$ defined by setting $h(z):= 1$ (where $1$ is the largest element of  $V$) if $z\not\in \{a,b\}$, $h(z)=r'$ if $z= a$ and $h(z):= 0$ (where $0$ is the least element of $V$) if $z=b$.  The intersection of balls $B(z', h_f(z'))$ of $E'$ contains $b'$, hence the map $h_f$ is not a hole of $E'$. Since $f$ is hole-preserving, $h$ is not a hole of $E$.  The intersection of  balls $B(z, h(z))$ of $E$ being  included into $\{b\}$ it is equal to $\{b\}$. Hence, $b\in B(a, r')$. It follows that $r'=r$.   Next, we prove that the range $A'$
 of $f$ is a $1$-local retract of $E'$. We apply Lemma \ref{lem:one-local-retract}. Let $B(x_i, r_i )$, $i\in I$,  where  $x_i \in A'$, $r_i\in V$,  be a family of balls of $E'$ such that the intersection over $E'$ is nonempty. Let $a'$ in this intersection and let $h:E\rightarrow V$ defined by setting $h(z):= d'(f(z), a')$. The intersection  of balls $B(z', h_f(z'))$ of $E'$ contains $a'$ hence $h_f$ it is not a hole of $E'$. Since $f$ is hole-preserving, $h$ is not a hole. Hence,  there is some  $a$ in the intersection  of balls $B(z, h(z))$ of $E$. Let $b':= f(a)$. Then $b'$ belongs to $\bigcap_{i\in I} B(x_i, r_i )$. The conclusion that $A'$ is a one-local retract follows from Lemma \ref{lem:one-local-retract}.  The converse is immediate. Let $h$ be a hole of $E$ and $h_f$ be its image. If $h_f$ is not a hole of $E'$, the intersection  of balls $B(z', h_f(z'))$ of $E'$ contains some element $a'$. Since $A'$ is a one-local retract  of $E'$ there is a retraction fixing $A'$ and sending $a'$ onto some element $b \in A'$. Let $a\in E$ such that $f(a)=b$.  Then $a\in \bigcap{z\in E}B(z, h(z))$.  Indeed, let $z\in E$; we claim that $d(z,a)\leq h(z)$ (indeed, since $f$ is an isometry, $d(z,a)= d'(f(z),b)$ and,  via the retraction,  $d'(f(z), b)\leq d'(f(z), a)\leq h_f(f(z)\leq h(z)$).   Hence $h$  is not a hole of $E$. Contradiction.  \end{proof}

Replacing isometries by hole-preserving maps in the definition of absolute retracts and injectives, we have the notions of absolute retracts and injectives w.r.t. holes preserving maps. 
 
We recall the following result of \cite{jawhari-al}.

\begin{theorem} On an  involutive Heyting algebra $V$,  the absolute retracts and the injective w.r.t. hole-preserving maps coincide. The class $\mathcal H$ of these objects is closed under product and retraction. Moreover, every metric space embeds into some member of $\mathcal H$ by some hole-preserving map. \end{theorem}

The  proof relies on  the introduction of the  replete space 
$H(E)$ of a metric space $E$. The space  $E$ is a absolute retract or not depending wether  $E$ is a retract of $H(E)$ or not. Furthermore, it allows to prove the \emph{transferability} of holes preserving maps, that is the fact that  for every  non-expansive map $f: E \rightarrow F$, hole-preserving map $g: E\rightarrow G$  there are hole-preserving map  $g' F\rightarrow E'$ and non-expansive map $f': G\rightarrow  E'$ such that  $g'\circ f= f'\circ g$.   Indeed, one may choose $E'= H(E)$. 

We give the key ingredients. 
 
 Let $E$ be a $V$-metric space. A \emph{ weak metric form} on $E$ is any map $h: E \rightarrow V$ such that $d(x,y) \leq h(x)\oplus \overline h(y)$ for all $x,y\in E$. If in addition, $h(x)\leq d(x,y)\oplus h(y)$  for all $x,y\in  E$, this is a \emph{metric form}. We denote by $C(E)$, resp. $L(E)$, the set of weak metric form, resp. of metric  forms.

Let  $H(E)$ the subset of $L(E)$ consisting of metric forms $h$ such that the intersection of balls $B(x, h(x))$ for $x\in E$ is nonempty. If  $V$ is an involutive Heyting algebra, we may equip $H(E)$ of the distance  induced by the sup-distance on $(V, d_V)^E$. We call it the \emph{replete space}.  

We recall the following result of \cite{jawhari-al}.

\begin{lemma}If $V$ is an involutive Heyting algebra then  $\overline \delta: E\rightarrow H(E)$  defined by $\overline \delta (x)(y):= d(y,x)$ is a hole-preserving map  from $E$ into $H(E)$. Futhermore $H(E)$ is an absolute retract w.r.t. the hole-preserving maps (i.e. this is a retract of every extension by a hole-preserving map).
\end{lemma}

\begin{problems} Let $E$ be a generalized metric space with  a compact and normal structure.
\begin{enumerate}[{(a)}]
\item  When  a one-local retract of $E$ is a retract? 

\item When  the set $Fix(f)$ of fixed point  of a non-expansive self map a retract? 

\end{enumerate}
\end{problems}

Note that if $(a)$ has a positive answer then spaces with a compact and normal structure are absolute retracts w.r.t.  hole-preserving maps. For these problems, it could be fruitful to consider the case of posets; there is a vast  literature on fix-point and this type of questions (see \cite{schroder, baclawski-bjorner, nevermann-wille}).

\subsection{The case of ordinary metric spaces} 
Let $\mathbb R^{+}$ be the set of non negative reals with the addition and natural order, the involution being the identity. Let $V:=\mathbb R^{+} \cup \{+\infty\}$. Extend to $V$ the addition and  order in a natural way. Then, metric spaces over $V$ are direct  sums of ordinary metric spaces (the distance between elements in different components being $+\infty$). The set $V$ is an involutive Heyting algebra, the distance $d_{V}$ once restricted to $\mathbb R^{+}$  is the absolute value. The inaccessible  elements are $0$ and $+\infty$ hence, if one  deals with ordinary metric spaces,  unbounded  spaces in the above sense are those which are unbounded in the ordinary sense.  If one deals with ordinary metric spaces, infinite products  can yield spaces  for which $+\infty$ is attained. On may replace powers with $\ell^{\infty}$-spaces (if $I$ is any set, $\ell^{\infty}_{\mathbb R} (I)$ is the set of families  $(x_i)_{i\in I}$ of reals numbers, endowed with the sup-distance). With that, the notions of  absolute retract, injective, hyperconvex and retract of some $\ell^{\infty}_{\mathbb R}(I)$ space coincide.

According to Corollary \ref{cor:compact+normal},  a  hyperconvex metric space   has a normal structure iff its diameter  is bounded. In fact, if a subset $A$ of a hyperconvex space is an intersection of balls, its radius is half the diameter. No description of metric spaces with a compact and normal structure seems to be known.

The existence of a fixed point for a non-expansive map on a bounded hyperconvex space is the famous result of Sine and Soardi. Theorem \ref{thm:cor} applied to a bounded  hyperconvex metric space is  Baillon's fixed point theorem.  Applied to 
a metric space with  a compact and normal structure,  this is the result obtained by the first author \cite{khamsi}.

\subsection{The case of ordered sets}\label{subsection:orderedset}
In this subsection, we consider  posets as binary relational systems as well as metric spaces over an involutive Heyting algebra.

Let $P:= (E, \leq)$ be an ordered set. 
Let $\mathcal E:= \{\leq, \leq^{-{1}}\}$ and $\mathbf E:= (E, \mathcal E)$. By definition, $\mathbf E$ is reflexive and involutive. For $x\in E$, set $\uparrow x:= \{y\in E: x\leq y\}$ and $\downarrow x:= \{y\in E: y \leq x\}$; this sets are called the \emph{principal} final, resp. \emph{initial},  segment  generated by $x$. With our terminology of balls of $\mathbf E$, these sets are the balls $B(x, \leq)$ and $B(x, \leq^{-1})$. 

Let $V$ be the following structure. The domain is the set  $\{0, +,-, 1\}$. The order is $0\leq a,b\leq 1$ with $+$ incomparable to $-$; the involution exchange $+$ and $-$ and fixes $0$ and $1$; the operation $\oplus$ is defined by $p\oplus q:= p\vee q$ for every $p,q \in V$. As it is easy to check, $V$ is an involutive Heyting algebra
  
If $(E, d)$ is a $V$-metric space, then $P_{d}:=(E, \delta_{a})$,   where $\delta_{+}:= \{(x,y): d(x,y)\leq +\}$,  is an ordered set. Conversely,  if $P:= (E, \leq )$ be an ordered set, then the map $d:E\times E\rightarrow V$ defined by $d(x,y):= 0$ if $x=y$,  $d(x,y):= +$ if $a<b$, $d(x,y):=-$ if $y<x$ and $d(x,y):= 1$ if $x$ and $y$ are incomparable. Clearly, if $(E,d)$ and $(E',d')$ are two $V$-metric spaces, a map $f:E\rightarrow E'$ is non-expansive from $(E,d)$ into $(E',d')$ iff it is order-preserving from $P_{d}$ into $P_{d'}$. 
Depending on the value of $v\in V$, a  $V$-metric space  has four types of balls: singletons, corresponding to  $v=0$, the full space, corresponding to   $v=1$, the principal  final segments,  $\uparrow x:= \{y\in E: x\leq y\}$, corresponding to balls  $B(x, +)$, and principal initial segments, $\downarrow  x:= \{y\in E : y\leq x\}$, corresponding to  balls   $B(x, -)$. 
The set  $V$ can be equipped with the distance $d_V$ given  by means of the formula (\ref{distance}). The corresponding poset is the four element lattice $\{-, 0, 1, +\}$ with $-< 0, 1< +$. The  retracts of powers of this lattice are all complete lattices. This is confirmed by the following fact. 

\begin{proposition}
A metric space $(E,d)$ over $V$ is hyperconvex iff the corresponding poset is a complete lattice. 
\end{proposition} 

\begin{proof} Suppose that $(E,d)$ is hyperconvex. Let $\leq:= \delta_{+}$ and $P_{d}:=(E, \delta_{+})$.  We prove that every subset $A$ has a supremum in $P_d$. This amounts to prove  that $A^{+}:= \{y\in E: x\leq y\;  \text{for all}\; x\in A\}$ has a least element. Since $(E,d)$ satisfies the convexity property, and $+\vee \overline {-}=1$, $B(x', +)\cap B(x'', +)\not= \emptyset$ for every $x',x'' \in E$; since $(E,d)$ satisfies the $2$-Helly property,  $A^{\Delta}= \bigcap _{x\in A}B(x, +)\not= \emptyset$. Applying again the convexity and $2$-Helly property, wet get that the intersection of balls $B(x,+)$ for $x\in A$ and $B(y,-)$, for $b\in A^{\Delta}$ is nonempty. This intersection contains just one element, this is the  supremum of $A$. A similar argument yields the existence of the infimum of $A$, hence $P_d$ is a complete lattice. Conversely, let $B(x_i,r_i)$, $(i\in I)$, be a family of balls such that $d(x_i, x_j)\leq r_i\vee \overline r_j$. We prove that $C:= \bigcap_{i\in I} B(x_i,r_i)\not =\emptyset$. If there is some $i\in I$ such that $r_i=0$, then $x_i\in C$. If not, let $A:= \{i\in I: r_i=+\}$, $B:=  \{j\in I: r_j= -\}$. Then $x_i\leq x_j $ for all $x_i\in A$, $x_j\in B$. Set $c:= \bigvee A$ and observe that $c\in C$. 
\end{proof}

Since $0$ is the only inacessible element of $V$, Theorem \ref{thm:cor3} applies: \emph{Every commuting family of order-preserving maps on a complete lattice has a common fixed point}. This is Tarski's theorem (in full). 

Posets coming from  $V$-metric spaces with a compact and normal structure are a bit more general than complete lattice, hence Theorem \ref{thm:cor} on compact normal structure could say a bit more than Tarski's fixed point theorem. In fact, for one order-preserving map,  this is no more as Abian-Brown's fixed-point theorem. 

Indeed, let us recall that a poset $P$ is \emph{chain-complete} if every nonempty chain in  $P$  has a supremum and an infimum.  

We prove below  that:

\begin{proposition}\label{lem:chain-complete}  If the collection of intersection of balls of a poset $P:= (E, \leq) $ satisfies the f.i.p.,  that is $\mathcal B_{E}$ is compact, then  $P$ is chain complete (converse false). \end{proposition}

Abian-Brown's theorem \cite {abian-brown} asserts that in a chain-complete poset with a least or largest element, every order-preserving map has a fixed point. 

The fact that the collection of intersection of balls of $P$ has a normal structure  means that every nonempty intersection of balls of $P$ has either a least or largest element. Being the intersection of the empty family of balls,  $P$ has either  a least element or a largest element. 

Consequently, if $P$ has a compact and normal structure, we may suppose without loss of generality  that it has a least element. Since every nonempty chain have a supremum, it follows from Abian-Brown's theorem that every order preserving map has a fixed point. 

On an other hand,  a description of posets with a compact and normal structure has yet to come. 

The proposition above follows from properties of gaps we rassemble below. %
%

A pair of subsets $(A, B)$ of  $E$ is called a \emph{gap} of $P$ if every element of $A$ is dominated by every element of $B$ but there is no element of $E$ which dominates every element of $A$ and is dominated by every element of $B$ (cf. \cite{duffus-rival}). In other words:  
$(\bigcap_{x\in A} B(x, \leq ))\cap (\bigcap_{y\in B} B(y, \geq ))= \emptyset$ while $B(x, \leq )\cap B(y, \geq ) \not =\emptyset$ for every $x\in A, y\in B$. A \emph{subgap} of $(A,B)$ is any pair $(A', B')$ with $A'\subseteq A$, $B'\subseteq B$, which is a gap. The gap $(A,B)$ is \emph{finite} if $A$ and $B$ are finite, otherwise it is \emph{infinite}.  Say that an ordered set $Q$ \emph{preserves} a gap $(A,B)$ of $P$ if there is an order-preserving map $g$ of $P$ to $Q$ such that $(g(A), g(B))$ is a gap of $Q$. On the preservation of gaps, see \cite{nevermann-wille}.

\begin{lemma}\label{lem:gap}
Let $P:= (E, \mathcal E)$ be a poset. Then:
\begin{enumerate}[{(a)}]
\item $P$ is  a complete lattice iff $P$ contains no gap; 
\item An order-preserving map $f:P\rightarrow Q$ is an embedding preserving all gaps of $P$ iff it preserves all holes of $P$ with values in $V\setminus \{0\}$ iff $f(P)$ is  a one-local retract of $Q$; 
\item $\mathcal B_{\mathbf E}$ satisfies the f.i.p. iff every gap of $P$ contains a finite subgap iff every hole is finite. 

\end{enumerate}

\end{lemma}
\begin{proof}
$(a)$. Let $(A, B)$ be a pair of subsets of $E$ such that every element of $A$ is dominated by every element of $B$.  Let  $A^{\Delta}:= \{y\in E: x\leq y \; \text{for all} \; x\in A\}$. Then, trivially,  $A\subseteq A^{\Delta}$ and every element of $A^{\nabla}$  dominates every element of $A$; furthermore $(A, A^{\Delta})$ is not a gap iff $A$ has a supremum. Thus,  if $P$ is a complete lattice, $A$ has a supremum, hence $(A, A^{\Delta})$ is not a gap and hence $(A,B)$ is not a gap. Conversely, if $P$ contains no gap, $(A,A^{\Delta})$ is not a gap and thus $A$ has a supremum. It follows that $P$ is a complete lattice.

$b)$. Suppose that $f$ is an embedding preserving all gaps. Let $h$ be a hole of $P$ with values in $V\setminus\{0\}$ and $h_f$  be its image. Let $A:= \{x\in P: h(x)=+\}$ and $B:= \{y\in P: h(y)=-\}$. If there is some $a\in A$, $b\in B$ such that $a\not \leq b$ then, since $f$ is an embedding, $f(a)\not \leq f(b)$ and  $h_f$ is  a hole of $Q$. Otherwise,  $B\subseteq A^{\Delta}$. Since $h$ is a hole in $P$, $(A,B)$ is a gap of $P$. Since $f$ preserves all gaps of $P$, $(f(A),f(B))$ is a gap of $Q$. It turns out that $h_f$ is a gap of $Q$. For the converse, let $(A,B)$ be a gap of $P$. We claim that $(f(A), f(B))$ is a gap of $Q$. Since $f$ is order preserving $(f(A)\subseteq f(B)^{\Delta}$. We only need to check that there is no element between $f(A)$ and $f(B)$. Let $h: P\rightarrow V\setminus \{0\}$ defined by setting $h(a):=+$ if $a\in A$,  $h(b):=-$ if $b\in B$ and $h(c)=1$ if $c\in E\setminus A\cup B$. Then, clearly, $h$ is a hole of $P$; since $f$ preserves it, $h_f$ is a hole of $Q$. Hence $\emptyset =\bigcap_{y \in Q} B(y, h_f(y))= \bigcap_{x\in A\cup B}B(f(x), h_f(x))$. If follows that $(f(A), f(B))$ is a hole of $Q$.  The equivalence with the last assertion is essentially Lemma \ref{lem:hole-localretract}.  

$(c)$. Suppose  that   $\mathcal B_{\mathbf E}$ satisfies the f.i.p.  Let$(A,B)$ be a gap. If  every finite pair $(A', B')$ with $A'\subseteq A$ and $B'\subseteq B$ is not a gap, then the finite intersections of $\uparrow a \cap \downarrow b$, with $a\in A$, $b\in B$ are nonempty. From the f.i.p. property, the whole intersection $\bigcap_{a\in A, b\in B}\uparrow a \cap \downarrow b$ is nonempty, contradicting the fact that $(A,B)$ is a gap. Conversely, let $\mathcal F$ be a family of members of $\mathcal B_{\mathbf E}$ whose finite intersections are nonempty. Each member of $\mathcal F$ being an intersection of balls,  each of the form $B(x, \leq)$  or $B(y, \geq)$,  we may in fact suppose that these  members are of the form $B(x, \leq)$  or $B(y, \geq)$. Hence, we may suppose that  there are two sets $A$ and $B$ such that $\mathcal F: =\{ B(x, \leq): x\in A\} \cup \{B(y, \geq ): y\in B\}$. Since  $(A,B)$  contains no finite gap, the pair $(A, B)$ is not a gap, hence $\bigcap \mathcal F\not = \emptyset$. The equivalence with the last assertion is Lemma \ref{lem:hole}. 
\end{proof}

We only mention some examples. 

Let  $\bigvee$ be the  $3$-element poset consisting of $0, +, -$ with $0< +, -$ and $+$ incomparable to $-$. We denote by $\bigwedge $ its dual. 
 Then the reader will observe that retracts of powers of $\bigvee$ have a compact and normal structure. 

 Theorem \ref{thm:cor}  above yields a fixed point theorem for a commuting family of order-preserving maps on any retract of   power of $\bigvee$ or of  power of  $\bigwedge$. But this result  says nothing about  retract of   products  of $\bigvee$ and $\bigwedge$.  
 
 These two posets fit in the category of fences. A \emph{fence} is a poset whose the comparability graph is a path. For example, a two-element chain is a fence. Each larger fence has two orientations, for example on the three vertices path, these orientations yield the $\bigvee$ and the $\bigwedge$. 
 
 From Theorem \ref{thm:cor4},  proved in  Subsection \ref{subsection:zigzag},  it  will follow:

\begin{theorem}\label{thm:cor5}
If a poset $Q$ is a retract of a product $P$ of   finite fences of bounded length, every commuting set of order-preserving maps on $Q$ has a fixed point.
\end{theorem}

Since every complete lattice is a retract of a power of the two-element chain, this result contains Tarski's fixed point theorem. 

\subsection{The case of oriented graphs}\label{subsection:orientedgraphs}

A \emph{directed graph} $G$ is a pair $(E, \mathcal E)$ where $\mathcal E$ is a binary relation on $E$. We say that  $G$ is  \emph{reflexive} if $\mathcal E$ is reflexive and that $G$ is \emph{oriented} if $\mathcal E$ is antisymmetric (that is $(x,y)$ and $(x,y)$ cannot be in $\mathcal E$ simultaneously except if $x =y$). If $\mathcal E$ is symmetric,  we identifies it with a subset of pairs of $E$ and we say that the graph is \emph{undirected}. 

If  $G:= (E, \mathcal E)$ and $G':= (E', \mathcal E')$ are two directed graphs,  an \emph{homomorphism from $G$ to $G'$} is a map $h: E\rightarrow  E'$ such that $(h(x), h(y)) \in \mathcal E'$ whenever $(x,y)\in \mathcal E$ for every $(x,y)\in E\times E$.

 Let us recall that a finite \emph{path} is an undirected graph $L:=(E, \mathcal E)$ such that one can enumerate the vertices into a non-repeating sequence  $v_0, \dots, v_n$ such that edges are the  pairs $\{v_i,v_{i+1}\}$ for $i<n$.  A \emph{reflexive zigzag} is a reflexive graph such that the symmetric hull is a path. If  $L$ is a reflexive oriented zigzag, we may enumerate the vertices in a non-repeating sequence $v_0:= x, \dots,  v_{n}:= y$ and to this enumeration we may  associate the finite sequence $ev(L):= \alpha_0\cdots \alpha_i \cdots\alpha_{n-1}$ of $+$ and $-$,   where $\alpha_i:= +$ if $(v_i,v_{i+1})$ is an edge and $\alpha_i:= -$ if $(v_{i+1},v_{i})$ is an edge. We call  such a sequence a \emph{word} over the \emph{alphabet} $\Lambda:= \{+,-\}$ , If the path has just one vertex, the corresponding  word  is   the empy word, that we denote by $\Box$.  Conversely, to a finite word  $u:= \alpha_0\cdots \alpha_i \cdots\alpha_{n-1}$ over $\Lambda$ we may associate the reflexive oriented zigzag $L_u:= (\{0, \dots n\}, \mathcal L_{u})$ with end-points $0$ and $n$ (where $n$ is the length $\ell(u)$ of $u$) such that  $\mathcal L_{u}= \{(i,i+1): \alpha_i=+\}\cup \{(i+1, i): \alpha_i=-\}\cup \Delta_{\{0, \dots, n\}}$. 
 
 \subsection{The zigzag distance}\label{subsection:zigzag}
 
 Let $G:= (E, \mathcal E)$ be a reflexive directed graph. For each pair $(x,y)\in E\times E$, the \emph{zigzag distance} from $x$ to $y$ is the set $d_G(x,y)$ of words $u$ such that there is a non-expansive map $h$ from $L_u$ into $G$ which send $0$ on $x$ and $\ell(u)$ on $y$.
 
 This notion is due to Quilliot \cite{Qu1, Qu2} (Quilliot considered reflexive directed graphs, not necessarily oriented, and in defining the distance, considered only  oriented paths). A general study is presented in \cite{jawhari-al}; some developments appear in \cite{Sa} and \cite{KP2}. 
 
 Because of the reflexivity of $G$, every word obtained from a word belonging to $d_G(x,y)$ by inserting letters  will be also into $d_G(x,y)$. This leads to  the following framework. 
 
 Let $\Lambda^*$ be  collection of words over the alphabet  $\Lambda:= \{+,-\}$. Extend the involution on $\Lambda$ to $\Lambda^*$ by setting $\overline \Box:= \Box$ and $\overline {u_0\cdots u_{n-1}}:= \overline{u_{n-1}}\cdots \overline{u_{0}}$ for every word in $\Lambda^*$. Order  $\Lambda^{\ast}$ by the \emph{subword ordering}, denoted by $\leq$. If  $u: = \alpha_{1} \alpha_{2} \ldots \alpha_{m},  v: = \beta_{1} 
\beta_{2} \ldots \beta_{n}\in \Lambda^{*}$ set  
$$u\leq v \; 
\text{if and only if }\; 
\alpha_{j}  =\beta_{i_{j}}\ {\rm for\ all}\ j = 1, \ldots m\; \text{with some}\; 1\leq j_1<\dots j_m\leq n.$$
Let $\mathbf {F}(\Lambda^*)$ be  the set of final segments of $\Lambda^*$, that is subsets $F$ of $\Lambda^{\ast}$ such that $u\in F$ and $u\leq v$ imply $v\in F$. 
Setting $\overline X:= \{\overline u: u\in X\}$ for a set $X$ of words, we observe that $\overline X$ belongs to  $\mathbf {F}(\Lambda^*)$.   
Order $\mathbf {F}(\Lambda^*)$  by reverse of the inclusion, denote by   $0$ its least element (that is $\Lambda^*$), set $X\oplus Y$ for the concatenation $X\cdot Y:= \{uv: u\in X, v\in Y\}$. 
Then,  one immediately see that $\mathcal H_{\Lambda}:= (\mathbf {F}(\Lambda^*), \oplus, \supseteq,  0, -)$ is an involutive Heyting algebra. This leads us to consider distances and metric spaces over $\mathcal H_{\Lambda}$. 
There are two simple and crucial facts about the consideration of the zigzag distance(see \cite {jawhari-al}).

\begin{lemma} A map from a reflexive directed graph $G$ into an other is a graph-homomorphism iff it is non-expansive.
\end{lemma}

\begin{lemma}
The distance $d$ of  a metric space $(E,d)$ over $\mathcal H_{\Lambda}$ is the zigzag distance of a reflexive directed graph $G:= (E, \mathcal E)$ iff  it satisfies the following property for all  $x,y,z \in E$, $u,v\in \mathbf {F}(\Lambda^*)$:
 $u.v \in d(x,y)$ implies $u\in d(x,z)$ and $v\in d(z,y)$ for some $z\in E$. When this condition holds, $(x,y)\in \mathcal E$ iff $+\in d(x,y)$. 
\end{lemma} 

Due to this later fact, the various metric spaces mentionned above (injective, absolute retracts, etc.) are graphs equipped with the zigzag distance; in particular, the distance $d_{\mathcal H_{\Lambda}}$  defined on $\mathcal H_{\lambda}$ is the zigzag distance of some graph. This facts leads to a fairly precise description of absolute retracts in the category of reflexive directed graphs (see \cite{KP2}). The situation  of oriented graphs is  different. These graphs cannot be modeled over a Heyting algebra (Theorem IV-3.1 of  \cite{jawhari-al} is erroneous), but the absolute retracts in this category can be (\cite{Sa}). The appropriate Heyting algebra is the  \emph{MacNeille completion} of $\Lambda^{\ast}$.

 The MacNeille completion is in some sense the least complete lattice extending $\Lambda^{\ast}$. The definition goes as follows. If $X$ is a subset of $\Lambda^{\ast}$ ordered by the subword ordering then 
 
 $$X^{\Delta}:= \bigcap_{x \in X} \uparrow x$$
is the {\it upper cone} generated by $X$, and 

$$X^{\nabla}:= \bigcap_{x \in X} \downarrow x$$
is the {\it lower cone} generated by $X$. 
The pair $(\Delta, \nabla)$ of mappings on the complete 
lattice of subsets of  $\Lambda^{\ast}$ constitutes a Galois connection. Thus,  a set $Y$ is a lower cone  if 
and only if $Y = Y^{\nabla \Delta}$, while a set $W$ is 
an upper cone if and only if $W = W^{\Delta \nabla}.$ This Galois connection 
$(\Delta, \nabla)$ yields the {\it Mac Neille completion} of 
$\Lambda^{\ast}.$ This completion is realized  as the complete 
lattice $\{W^{\nabla}:  W\subseteq \Lambda^{\ast}\}$ 
ordered by inclusion or $\{Y ^{\Delta} : Y\subseteq \Lambda^{\ast}\}$ ordered by reverse inclusion. In this paper,  we choose as completion the set $\{Y ^{\Delta} : Y\subseteq \Lambda^{\ast}\}$ ordered by reverse inclusion  that we denote by $\mathbf {N}(\Lambda^{\ast})$. This complete lattice is studied in details in \cite{bandelt-pouzet}.

We recall the important fact that sets of the form  $W^{\nabla}$ for $W$ nonempty coincide with nonempty finitely generated initial segments of $\Lambda^{\ast}$ (Jullien \cite{jullien}). Hence:
\begin{lemma}\label{lem:jullien} The set $\mathbf {N}(\Lambda^{\ast})\setminus \{\emptyset\}$ is order isomorphic to the set $\mathbf I_{<\omega}(\Lambda^{\ast})\setminus \{\emptyset\}$ ordered by inclusion  and made of finitely generated initial segments of $\Lambda^{\ast}$. In particular, $\mathbf {N}(\Lambda^{\ast})\setminus \{\emptyset\}$ is a distributive lattice. \end{lemma}

 The concatenation, order and involution defined on $\mathbf {F}(\Lambda^{\ast})$ induce a involutive  Heyting algebra $\mathcal N_{\Lambda}$ on  $\mathbf {N}(\Lambda^{\ast})$ (see Proposition 2.2 of \cite{bandelt-pouzet}). Being an involutive Heyting algebra, $\mathcal N_{\Lambda}$ supports a  distance $d_{\mathcal N_{\Lambda}}$ and this distance is the zigzag distance of a graph $G_{\mathcal {N}_{\Lambda}}$. But it is not true that every oriented graph embeds isometrically into a power of that graph. For example, an oriented cycle cannot.  The following result characterizes  graphs which can be  isometrically embedded, via the zigzag distance, into  products of reflexive and oriented zigzags. It is stated in part in Subsection IV-4 of \cite{jawhari-al}, cf. Proposition IV-4.1.

\begin{theorem}\label{theo:isometric}
	For a  directed graph   $G : = (E, \mathcal E)$ equipped with the zigzag distance,  the following properties are equivalent:
\begin{enumerate} [(i)] 
	\item $G$ is isometrically embeddable into a product of  reflexive and oriented zigzags; 
	\item $G$ is isometrically embeddable into a power of $G_{\mathcal  N_{\Lambda}}$; 
	\item The values of the zigzag distance between  vertices of $E$ belong to $\mathcal N_{\Lambda}$. 
\end{enumerate}
\end{theorem}

\begin{proof} $(i)\Rightarrow (ii) \Rightarrow (iii)\Rightarrow (i)$.

$(i)\Rightarrow (ii)$. The proof relies on the following:

\begin{claim}\label{claim-zigzag}Every finite reflexive oriented zigzag is isometrically embeddable into $G_{\mathcal  N_{\Lambda}}$. 
\end{claim}

\noindent{\bf Proof of Claim \ref{claim-zigzag}.}
Let $L$ be a finite reflexive oriented zigzag. Let $n$ be its number of vertices. There is a word $u:=\alpha_0\cdots \alpha_i \cdots\alpha_{n-1}\in \Lambda^{\ast}$ such that $L$ is isomorphic to $L_u:= (\{0, \dots n\}, \mathcal L_{u})$. Let $\varphi :\{0,\dots, n\}\rightarrow N(\Lambda^{\ast})$ be the map defined by $\varphi (i):= \uparrow u_{<i}$ where  $u_{<i}:= \Box$ if $i=0$ and $u_{<i}:= \alpha_0\cdots \alpha_{i-1}$ otherwise. We claim that $\varphi$ is an isometry from $L$ equipped with the zigzag distance into $(\mathcal N_{\Lambda}, d_{\mathcal N_{\Lambda}})$, that is $d_{L}(i,j)= d_{\mathcal N_{\Lambda}}(\varphi(i), \varphi(j))$ for all $i,j\leq n$. It suffices to check that this equality holds for $i<j$. In this case, $d_{L}(i,j)= \uparrow \alpha_i\cdots \alpha_{j-1}$. In $(\mathcal N_{\Lambda}, d_{\mathcal N_{\Lambda}})$,  we have: 

\begin{equation}
d_{\mathcal N_{\Lambda}}(v,v\oplus w)=w , \end{equation} 

for all $v\in \mathcal N_{\Lambda}$, $w\in \mathcal N_{\Lambda}\setminus \{\emptyset\}$.

Indeed, due to the definition of the distance in $\mathcal N_{\Lambda}$, we have $u+ d_{\mathcal N_{\Lambda}}(v,v\oplus w)= u+w$.  As a monoid,  $\mathcal N_{\Lambda}\setminus \{\emptyset\}$ is cancellative (see Lemma 11 of  \cite{KPR}). Hence $d_{\mathcal N_{\Lambda}}(v,v\oplus w)=w$. Thus, $d_{\mathcal N_{\Lambda}}(\varphi (i),\varphi(j))=d_{\mathcal N_{\Lambda}}(\varphi (i),\varphi(i)\oplus \uparrow \alpha_i\cdots \alpha_{j-1})= \uparrow \alpha_i\cdots \alpha_{j-1}=d_{L}(i,j)$, as required. Since $(\mathcal N_{\Lambda}, d_{\mathcal N_{\Lambda}})$ is hyperconvex, the distance $d_{\mathcal N_{\Lambda}}$ is the zigzag distance associated to the oriented graph $G_{\mathcal  N_{\Lambda}}$, hence the isometric  embedding $\varphi$ induces a graph embedding.
\hfill $\Box$

With Claim \ref{claim-zigzag}  we may embed isometrically any  product of zigzags into a power of $G_{\mathcal N_{\Lambda}}$. This proves that $(ii)$ holds. 

$(ii)\Rightarrow (iii)$. If $G'$ is a product of graphs $G'_i$ , the zizag distance on $G'$ is the sup-distance on the product of the metric spaces $(G_i, d_{G_i})$. Thus, if $G$ isometrically embeds into a power of  $G_{\mathcal N_{\Lambda}}$, $(G, d_G)$ isometrically embeds into a power of $(\mathcal N_{\Lambda}, d_{\mathcal N_{\Lambda}})$. Since the distance $d_{\mathcal N_{\Lambda}}$ has values in $\mathcal N_{\lambda}$, $d_G$ has values in $\mathcal N_{\Lambda}$ too hence $(iii)$ holds.

$(iii)\Rightarrow (i)$.
 The proof follows the same lines as the proof of  Proposition IV-4.1 p.212 of \cite{jawhari-al}. 

We use the following property: 
\begin{claim}\label{claim-extension property} For each pair  of vertices   $x,  y \in E$  and each word  $u \in (d_G(x, y))^{\nabla}$, let  $L_{u}$ be reflexive oriented path 
with end points $0$ and $\ell(u)$ associated with $u$. The map carrying $x$ onto $0$ and $y$ onto $\ell(u)$ extends to a non-expansive mapping  $f_{x,y,u}$ from $G$ onto  $L_u$. 
\end{claim}

\noindent{\bf Proof of Claim \ref{claim-extension property}.}
The proof of the claim relies onto two facts. First,  $d_{L_u}(0, \ell(u))=\uparrow u$. Since  $u\in d_G(x,y)$, $\uparrow u\leq d_G(x,y)$, hence  the partial map carrying $x$ onto $0$ and $y$ onto $\ell(u)$ is a non-expansive map  from the subset $\{x,y\}$ of $G$ equipped with the zigzag distance $d_{G}$ into the space associated to the zigzag $L_u$. 
Next, such a partial map extends to $G$ to a non-expansive mapping. This is due to the fact that the space associated to $L_u$  is hyperconvex  (it is trivially convex and since  each ball in that space is an interval of its domain $\{0, \dots,  \ell(u)\}$,  any collection of balls has the $2$-Helly-property). For the fact that non-expansive maps with values into an hyperconvex space extend, see \cite{jawhari-al}. \hfill $\Box$

 Let 
$$G':= \Pi\{L_u: u\in (d_G(x, y))^{\nabla}\;  \text {and}\;  (x, y) \in E\times E \}.$$ 
For each $x,  y \in E$  and each word  $u \in (d_G(x, y))^{\nabla}$, let  $f_{x,y,u}$ be a non expansive mapping from $G$ onto  $L_u$. 
We claim that the graph  $G$ is isometrically embeddable into  $G'$ by the map  $f$ defined by setting for every $z\in E$:
$$f(z):= \{f_{x,y,u}(z): u\in (d_G(x, y))^{\nabla}\;  \text {and}\;  (x, y) \in E\times E\}.$$  

 This map is  an isometry; indeed first, by definition of the product, it is non-expansive; next, to conclude that it is an isometry, it suffices to  check that for every $v \in \Lambda^*$, if  $d_G(x, y)\not \leq \uparrow v$ then  $d_{G'}(f(x), f(y)\not \leq \uparrow v$, that is for some triple $i:= (x',y',u)$ one has $d_{G_i}(f_i(x), f_i(y))\not \leq \uparrow v$. Let $v$ and  $x, y$ such that   $d_G(x, y)\not \leq  \uparrow v $ (this amounts to  $v \not \in d_G(x, y)$). Since  $d_G(x, y)  = ((d_G(x, y))^{\nabla})^{\Delta}$ there is some   $u\in (d_G(x, y))^{\nabla}$  such that 
$u \not \leq  v$. We may set $i:=(x, y, u)$. 

We may note that  the product can be infinite  even if the graph $G$ is finite. Indeed, if $G$ consists of two  vertices $x$ and $y$ with no value on the pair $\{x, y\}$ (that is the underlying graph is disconnected) then we need infinitely many zigzags of arbitrarily long length.    

\end{proof}

\begin{lemma} \label{lem-accessible}Every element $v$ of $\mathcal N_{\Lambda}\setminus \{\Lambda^{\ast}, \emptyset\}$ is accessible. 
\end{lemma}
\begin{proof}
Case $1$.  $v= \uparrow u$. Then $n:=\ell(u)\not =0$ hence $u= \alpha_0\cdots \alpha_{n-1}$. Set $u':=\alpha_0\cdots \alpha_{n-2}\overline {\alpha_{n-1}}$ and $r:= \uparrow u'$. Since $u\not \leq u'$, $v\not\leq r$. On an other hand $u\leq u'\oplus \overline {u'}$ hence $v= \uparrow u\leq \uparrow u'\oplus \overline {u'}= (\uparrow u')\oplus (\uparrow \overline {u'})= r\oplus \overline r$. Hence $v$ is accessible. 

Case $2$. If $v$ is not of the form $\uparrow u$ for some   $u\in \Lambda^{\ast}$. Since $u$  is not the emptyset, it is a finite join of elements of the form $\uparrow u$. Thus, we may suppose that $v= v_1 \vee  v_2$ where $v_1= \uparrow u_1<v$ and $v_2<v$ and furthermore that $v'_1\vee v_2<v$ for all $v'_1 <v_1$.  According to Case $1$, there is some $r_1$ such that $v_1 \not \leq r_1$ and $v_1\leq r_1 \oplus \overline r_1$. Let $r:= r_1 \vee v_2$. We claim first that $v\not\leq r$. Suppose the contrary, according to Lemma \ref{lem:jullien},  $\mathcal N_{\Lambda}\setminus \{\emptyset\}$ is a distributive lattice, hence  from $v\leq r$ we get $v= v\wedge r = (v\wedge r_1) \vee (v\wedge v_2)= (v\wedge r_1)\vee v_2$, contradicting the choice of  $v_1$. Next, we claim that $v\leq r\oplus \overline r$. The operation $\oplus$ and $\vee$ distribute (Theorem 10 in \cite{KPR}), hence $r\oplus  \overline r= (r_1\vee v_2)\oplus (\overline{r_1\vee v_2})= (r_1\vee v_2)\oplus (\overline{r_1} \vee \overline {v_2})= (r_1\oplus \overline r_1) \vee (v_2\oplus \overline r_1)\vee (r_1\oplus \overline v_2)\vee (v_2\oplus \overline v_2)$. Since $v_1\leq r_1\oplus \overline  r_2$, it follows that $v\leq r\oplus  \overline r$. Hence $v$ is accessible. 
\end{proof}

\begin{theorem}\label{thm:cor4}
If a graph $G$, finite or not, is a retract of a product of reflexive and directed zigags of bounded length then every commuting set of endomorphisms has a common fixed point. 
\end{theorem}
\begin{proof}
We may suppose that $G$ has more than one vertex. The diameter of $G$ equipped with the zigzag distance belongs to $\mathcal N_{\Lambda}\setminus \{\Lambda^{\ast}, \emptyset\}$. According to  Lemma \ref{lem-accessible},  it is accessible, hence as a metric space $G$ is bounded. Being a retracts of a product of hyperconvex metric spaces it is hyperconvex. Theorem \ref{thm:cor3} applies.   
\end{proof}

\subsection{Bibliographical comments}Generalizations of the notion of metric space  are  as old as the notion of ordinary metric space  and arises from geometry, logic as well as probability. Ours, originating  in  \cite{jawhari-al},  is one among several; the paper \cite{jawhari-al}
contains  $71$  references, e.g. Blumenthal and Menger \cite{blumenthal}, \cite{blumenthal1} \cite{blumenthal-menger}, as well as Lawvere \cite{lawvere}, to mention just a few. It was motivated by the work of Quilliot on graphs and posets \cite{Qu1,Qu2}. It extended to metric spaces over an involutive Heyting algebra (more appropriately an involutive  op-quantale) the characterization of hyperconvex spaces due to Aronszjan-Panitchpakdi \cite {aronszajn} and the existence of injective envelope, obtained for ordinary metric spaces by Isbell \cite{isbell}. It  contained also a study of hole-preserving maps and  a characterization of absolute retracts w.r.t. these maps by means of the replete space.  For more recent developments, see \cite{Abu-Sbeih-khamsi, bandelt-pouzet, KP1, KP2, KPR}.


\begin{thebibliography}{99}
\bibitem {abian-brown} Abian, S., Brown, A., 
A theorem on partially ordered sets, with applications to fixed point theorems.
Canad. J. Math. 13 1961 78- 82. 
\bibitem{Abu-Sbeih-khamsi}
Abu-Sbeih, M. Z.,  Khamsi, M. A. 
Fixed point theory in ordered sets from the metric point of view. Topics in fixed point theory, 223--236, Springer, Cham, 2014.
\bibitem {aronszajn}  Aronszajn N. and Panitchpakdi, P.,  Extensions of
uniformly continuous transformations and hyperconvex metric spaces, Pac. J. Math., 6(1956),405-439.

\bibitem {baclawski-bjorner}Baclawski, K., A.Bj\"ošrner, A., Fixed points in partially ordered sets, Advances in Math. 31,(1979)  263-287.

\bibitem {baillon}Baillon, J.B.,  Non expansive mapping and hyperconvex spaces. Fixed point theory and its applications (Berkeley, CA, 1986), 11-19, Contemp. Math., 72, Amer. Math. Soc., Providence, RI, 1988. 

\bibitem{bandelt-pouzet} Bandelt, H-J., Pouzet, M.,  {\em The MacNeille completion of the free monoid over an ordered alphabet}, preprint pdf, 23pp, Feb.  2006.
\bibitem{blumenthal} L.M.~Blumenthal, {\em Boolean geometry,} I. Rend. Circ. Mat. Palermo (2) 1 (1952), 343--360.
%
 \bibitem{blumenthal1} Blumenthal, L.M.,   {\em Theory and applications of distance geometry,}  Second edition Chelsea Publishing Co., New York 1970 xi+347 pp.
\bibitem{blumenthal-menger} Blumenthal, L.M., Menger, K., {\em Studies in geometry,} W. H. Freeman and Co., San Francisco, Calif. 1970 xiv+512 pp.

\bibitem	{bruck-all}Bruck, R.E. , A common fixed point theorem for a commuting family of non-expansive mappings, Pac. J. Math., 53 (1974), 59-71.

\bibitem{cohn}  Cohn, P. M. Universal algebra. Harper \& Row, Publishers, New York-London 1965 xv+333 pp.

\bibitem	{demarr}DeMarr, R., Common fixed point theorem for commuting contraction mappings, Pac. J. of Math., 13 (1963), 1139-1141.

 \bibitem{deza-deza}Deza, M.,  Deza, E.,   {\em Encyclopedia of distances},  Fourth edition. Springer, Heidelberg, 2016,  xxii+756 pp.
%
%
\bibitem{dress} Dress, A.W.N,   Trees, tight extensions of metric spaces, and the cohomological dimension of certain groups, a note on combinatorial properties of metric spaces,  Adv. in
Math. 53 (3)(1984), 321--402.

\bibitem{duffus-rival}  Duffus, D.,  Rival, I., structure theory for ordered sets. Discrete Math. 35 (1981), 53-118.
\bibitem {isbell} Isbell, J.R., 
 Six theorems about injective metric spaces,
Comment. Math. Helv. 39 1964 65-76.




\bibitem{jawhari-al} Jawhari, E., Misane, D., Pouzet, M., Retracts: Graphs and ordered sets from the metric point of view, in "Combinatorics and ordered sets" I.Rival ed., Contemporary Math. Vol 57 (1986),175-226. 


 \bibitem{jullien} Jullien, P.,   Sur un th\'eor\`eme d'extension dans la th\'eorie des mots. C. R. Acad. Sci. Paris S\'er. A-B 266 1968 A851--A854. 


\bibitem{KaPo1} Kabil, M., Pouzet, M.,  Une extension 
d'un
th\'eor\`eme de P.Jullien sur les \^ages de mots, Theoretical Informatics
and Applications. Vol 26, n$^\circ$ 5, (1992), 449-482.



\bibitem{KP1} Kabil, M., Pouzet, M., Ind\'ecomposabilit\'e et
irr\'eductibilit\'e dans la vari\'et\'e des r\'etractes absolus
des graphes r\'eflexifs,  C.R.Acad.Sci. Paris  S\'erie A  321
 (1995), 499-504.

\bibitem{KP2}Kabil, M., Pouzet, M.,  Injective envelope of graphs and
transition systems, Discrete Math.192 (1998), 145-186.

\bibitem{KPR}Kabil, M., Pouzet, Rosenberg, I.G, Free monoids and   metric spaces, To the memory of Michel Deza, 15 p.  March 2017. To appear in Europ. J. of Combinatorics.  arXiv:   1705.09750v1, 27 May 2017. 


\bibitem {khamsi}Khamsi, M.  A., One-local retract and common fixed point for commuting mappings in metric spaces,  Nonlinear Anal. 27 (1996), no. 11, 1307--1313. 
\bibitem {kirk}Kirk, W.A., A fixed point for mappings which do not increase distances, Amer. Math. Monthly, 72 (1965), 1004-1006.
\bibitem{lau} Lau, D.,  Function algebras on finite sets. A basic course on many-valued logic and clone theory. Springer Monographs in Mathematics. Springer-Verlag, Berlin, 2006. xiv+668 
\bibitem{lawvere}Lawvere, F.W., Metric spaces, generalized logic, and closed categories, Rend. Sem. Mat. Fis. Milano 43
(1973), 135--166 (1974).
	
\bibitem	{lim}Lim, T.C., A fixed point theorem for families of non-expansive mappings, Pacific J. of Math., 53 (1974), 487-493.
\bibitem {misane} Misane, D., 	R\'etracts absolus d'ensembles ordonn\'ees et de graphes. Propri\'et\'es du point fixe.             
	Th\`ese de 3\`eme cycle, Universit\'e Claude-Bernard, 14 septembre 1984.
\bibitem{nevermann-wille}Nevermann, P., Wille, R., The strong selection property and ordered sets of finite length, Algebra Universalis, 18 (1984). 18-28. 
\bibitem {penot} Penot, J.P., Fixed point theorems without convexity, in Analyse non convexe (1977, PAU), Bull. Soc. Math. France, M\'eŽmoire 60, (1979), 129-152.


\bibitem{penot}  Penot, J-P., Une vue simplifi\'ee de la th\'eorie de la complexit\'e. Gaz. Math. No. 34 (1987), 61-77.

 \bibitem{pouzet}Pouzet, M.,  Une approche m\'etrique de la r\'etraction dans les ensembles ordonn\'es et les graphes.  Proceedings of the conference on infinitistic mathematics (Lyon, 1984), 59--89, Publ. D\'ep. Math. Nouvelle S\'er. B, 85-2, Univ. Claude-Bernard, Lyon, 1985. 
 
 \bibitem{pouzet-rosenberg} Pouzet, M., Rosenberg, I.G., General metrics and contracting operations, in  Graphs and combinatorics (Lyon, 1987; Montreal, PQ, 1988). Discrete Math. 130 (1994),103--169.
 
 
 \bibitem{Qu1} Quilliot, A., Homomorphismes, points fixes,
r\'etractions et jeux de poursuite dans les graphes, les ensembles
ordonn\'es et les espaces m\'etriques, Th\`ese de doctorat d'Etat,
Univ Paris VI (1983).
\bibitem{Qu2} Quilliot, A., An application of the Helly
property
to the partially ordered sets, J. Combin. Theory, serie $A$, 35 (1983),
185-198.
\bibitem{rosenthal} Rosenthal, K.I., Quantales and their applications,  Pitman Research Notes in Mathematics Series, 234. Longman Scientific \& Technical, Harlow; copublished in the United States with John Wiley \& Sons, Inc., New York, 1990. x+165 pp.

\bibitem{Sa} Sa\"{\i}dane, S.,   Graphes et langages: une approche
m\'etrique,  Th\`{e}se de doctorat, Universit\'e Claude-Bernard, Lyon1, 14 Novembre 1991.
\bibitem {schroder} Schr\"oder, B., 
Ordered sets.
An introduction with connections from combinatorics to topology. Second edition. Birkhäuser/Springer, 2016. xvi+420 pp.
\bibitem {sine} Sine, R.C., On nonlinear contractions in sup norm spaces, Nonlinear Analysis, 3 (1979), 885-890.

\bibitem{snow}Snow, J.W,  A constructive approach to the finite congruence lattice representation problem. Algebra Universalis 43 (2000), no. 2-3, 279--293.

\bibitem{soardi} Soardi, P., Existence of fixed points of non-expansive mappings in certain Banach lattices, Proc. A.M.S., 73 (1979), 25-29.

\bibitem{tarski}Tarski, A., A lattice theoretical fixed point theorem and its applications, Pac. J. of Math. 5      (1955), 285-309.


 \end{thebibliography}
\end{document}